\documentclass[12pt,reqno]{amsart}
\usepackage{amssymb}
\usepackage[all]{xy}
\usepackage{color}

\allowdisplaybreaks
\evensidemargin=\oddsidemargin
\newtheorem{thm}{Theorem}[section]
\newtheorem{prop}[thm]{Proposition}
\newtheorem{cor}[thm]{Corollary}
\newtheorem{lemma}[thm]{Lemma}

\theoremstyle{remark}

\newtheorem{remark}[thm]{Remark}
\newtheorem*{sugcon*}{Convention}

\numberwithin{equation}{section}

\newcommand\cscatnd{%
\ensuremath{\operatorname{\mathtt{C*}}_{\!\!\!\!\mathtt{nd}}}}

\newcommand\cscoactndr{%
\ensuremath{\operatorname{\mathtt{C*coact}}_{\mathtt{nd}}^{\mathtt{r}}}}
\newcommand\csactnd{%
\ensuremath{\operatorname{\mathtt{C*act}}_{\mathtt{nd}}}}
\newcommand\cscoactndn{%
\ensuremath{\operatorname{\mathtt{C*coact}}_{\mathtt{nd}}^{\mathtt{n}}}}
\newcommand\cscoactnd{%
\ensuremath{\operatorname{\mathtt{C*coact}}_{\mathtt{nd}}}}

\newcommand\cscat{\ensuremath{\operatorname{\mathtt{C*}}}}

\newcommand\CP{\operatorname{\mathtt{CP}}}

\newcommand\CPn{\operatorname{\mathtt{CP}^{\mathtt{n}}}}
\newcommand\Fix{\operatorname{\mathtt{Fix}}}

\newcommand\Nor{\operatorname{\mathtt{Nor}}}
\newcommand\Inc{\operatorname{\mathtt{Inc}}}
\newcommand\Id{\operatorname{\mathtt{Id}}}
\DeclareMathOperator{\Aut}{Aut}

\newcommand{\quigg}{\delta^{Q}}
\DeclareMathOperator{\id}{id}
\DeclareMathOperator{\rt}{rt}

\newcommand\C{\mathbb{C}}

\newcommand\K{\mathcal{K}}
\newcommand\HH{\mathcal{H}}

\DeclareMathOperator{\Ad}{Ad}
\DeclareMathOperator{\clsp}{\overline{span}}
\DeclareMathOperator{\lsp}{{span}}

\DeclareMathOperator{\Red}{Red}
\DeclareMathOperator{\Mor}{Mor}

\let\phi\varphi

\def\new{\textcolor{black}}

\begin{document}

\begin{abstract}We survey the results required to pass between full and reduced coactions of locally compact groups on $C^*$-algebras, which say, roughly speaking, that one can always do so without changing the crossed-product $C^*$-algebra. Wherever possible we use definitions and constructions that are well-documented and  accessible to non-experts, and otherwise we provide full details. We then give a series of applications to illustrate the use of these techniques. We obtain in particular a  new version of Mansfield's imprimitivity theorem for full coactions, and prove that it gives a natural isomorphism between crossed-product functors defined on appropriate categories. 
\end{abstract}


\title[Full and reduced coactions]{\boldmath{Full and reduced coactions\\ of locally compact groups on $C^*$-algebras}}

\author[an Huef]{Astrid an Huef}
\address{Department of Mathematics and Statistics\\
University of Otago\\PO Box 56\\ Dunedin 9054\\
New Zealand}
\email{astrid@maths.otago.ac.nz}

\author[Quigg]{John Quigg}
\address{School of Mathematical and Statistical Sciences\\
Arizona State University\\Tempe\\
AZ 85287\\
USA}
\email{quigg@asu.edu}

\author[Raeburn]{Iain Raeburn}
\address{School  of Mathematics and Applied Statistics, University of
Wollongong, NSW 2522, Australia}
\email{raeburn@uow.edu.au}

\author[Williams]{Dana P. Williams}
\address{Department of Mathematics\\Dartmouth College\\Hanover, NH 03755\\USA}
\email{dana.williams@dartmouth.edu}

\date{January 3, 2010}

\thanks{This research was supported by the Australian Research Council}
\maketitle

\section{Introduction}

Every locally compact group has two $C^*$-algebras: the full group $C^*$-algebra $C^*(G)$, which is generated by a universal unitary representation of $G$, and the reduced group $C^*$-algebra $C_r^*(G)$, which acts faithfully via the left-regular representation $\lambda$ on $L^2(G)$. These algebras carry canonical comultiplications $\delta_G$ and $\delta_G^r$. A full coaction of $G$ on a $C^*$-algebra $B$ is an injective homomorphism $\delta:B\to M(B\otimes C^*(G))$, and  satisfies the coaction identity $(\delta\otimes{\id})\circ\delta=({\id}\otimes\delta_G)\circ \delta$; a reduced coaction is an injective homomorphism $\delta:B\to M(B\otimes C^*_r(G))$, and satisfies a similar identity with $\delta_G$ replaced by $\delta_G^r$. The key examples are the dual coaction $\hat\alpha$ on a full crossed product $A\times_\alpha G$, which is a full coaction, and the dual coaction $\hat{\alpha}^r$ on a reduced crossed product $A\times_{\alpha,r}G$, which is a reduced coaction.

In the early papers about coactions on $C^*$-algebras, authors used reduced coactions and spatial arguments (see, for example, \cite{L}, \cite{K} and \cite{LPRS}), and full coactions appeared rather later \cite{rae, Q2}. Full coactions did not immediately catch on, and researchers continued to study reduced coactions throughout the 1990s. So those interested in full coactions have to be able to convert results about reduced coactions into ones about full coactions. The cognoscenti generally agree that the results in \cite{rae} and \cite{Q2} allow one to do this, especially for questions about covariant representations and crossed products, but often key pieces of the puzzle are not readily available. Our goal here is to provide a comprehensive set of results \new{that} will allow researchers to pass as easily as possible between full and reduced coactions (both ways). We have tried whenever possible to choose definitions and constructions \new{that} are well-documented and  accessible to non-experts.

In \cite{rae}, Raeburn described a canonical procedure for passing from a full coaction $(B,\delta)$ to a reduced coaction $(B^r,\delta^r)$ on a quotient of $B$, and showed that this ``reduction process'' does not change the crossed product \cite[Theorem~4.1]{rae}. Subsequently, Quigg showed in \cite{Q2} that the same thing could be achieved in a two-step process: first ``normalise'' to get a full coaction $\delta^n$ on a quotient $B^n$ of $B$, and then reduce $\delta^n$ to get a reduced coaction on the same quotient $B^n$.  Quigg's approach has the advantage that the second step is reversible: provided we restrict attention to Quigg's class of normal full coactions, we can pass to and fro freely. This is the content of our \S\ref{nor=red}; most of the key results already exist in \cite{Q2} and \cite{kqrproper}, and we have merely gathered them together, tidied them up, and filled in some minor gaps. To completely understand the reduction process, though, we also need to understand Quigg's normalisation process, and this is the content of our \S\ref{secnormalising}.

We are thinking of the results of \S\ref{nor=red} and \S\ref{secnormalising} as providing ``boilerplates'', which can be bolted to papers on one kind of coaction and thereby allow users to apply these results to the other kinds. Each boilerplate consists of a Proposition \new{describing} the main properties of the construction (Propositions~\ref{propred} and~\ref{propnormal}), and a Theorem \new{describing} the functorial properties of the construction (Theorems~\ref{bashatbplate} and~\ref{bplatenor}). We have tested these boilerplates by applying them to a variety of theorems about crossed products (see \S\ref{secappl}). 

We begin with a short section in which we discuss our conventions on crossed products. This is an important issue in this paper: we want  constructions which give concrete $C^*$-algebras so that we can use them as the object maps in functors.\footnote{There is an alternative view that one does not need to do this and therefore shouldn't: one can instead define a crossed product to be a $C^*$-algebra with certain universal properties, and then use some generalisation of the axiom of choice to choose one crossed product for each system. This point of view is discussed in \cite{kqcat}.} For actions we define the crossed product as a completion, and for a coaction $(B,\delta)$ we define $B\times_\delta G$ to be a 
particular $C^*$-subalgebra of $B\otimes \K(L^2(G))$, and we call this the standard crossed product.

As we discussed above, the boilerplates themselves are in \S\ref{nor=red} and \S\ref{secnormalising}, and \S\ref{secappl} contains some applications. Our first application makes precise the assertion that ``reducing and normalising do not affect the crossed product'' by constructing crossed product functors and proving that the reduction and normalisation functors intertwine these crossed-product functors (\S \ref{secCP}). Next we extend a uniqueness theorem for crossed products by coactions from the version for reduced crossed products in \cite{rae} to full crossed products (see \S\ref{secdiut}). We use this new ``dual-invariant uniqueness theorem'' several times in our later applications. In \S\ref{secland}, we formulate four alternative versions of Landstad duality for actions. The original was proved in \cite{L}, and one of the other versions has been previously used in \cite{kqcat}, but we find it quite satisfying to see how the relationships between the different formulations fall out from the boilerplate. In \S\ref{subsecMansfield}, we discuss the relationship between the different versions of Mansfield's imprimitivity theorem for full and reduced coactions; we believe that our formulation of the theorem for full coactions is new. Our final application extends the results in \cite{kqrproper} on the naturality of Mansfield imprimitivity for reduced coactions to full coactions (see~\S\ref{secmannat}).  

Because coactions take values in multiplier algebras, the coaction identity
\begin{equation}\label{coactid}
(\delta\otimes{\id})\circ\delta=({\id} \otimes \delta_G)\circ \delta
\end{equation}
implicitly involves the extensions of the homomorphisms $\delta\otimes{\id}$ and ${\id} \otimes \delta_G$ to the multiplier algebra $M(B\otimes C^*(G)\otimes C^*(G))$. In early papers such as \cite{LPRS}\footnote{Though not \cite{L}, where the coactions were defined by restricting coactions on von Neumann algebras, and hence arrived extended.}, the extension process was explicitly announced by inserting bars, but people got bored with this and started leaving the bars off (see \cite{rae} and \cite{Q2}, for example). This is  okay so long as we remember that the process has its subtleties, and in particular that $\overline{\delta\otimes{\id}}$ is not the usual tensor product of two homomorphisms. After a few scares, we have vowed to be more careful in the future, and have included an appendix in which we discuss some of the subtleties involved in barring tensor products.

Fans of the relatively new maximal coactions introduced in \cite{ekq} will want yet another boilerplate. Each full coaction has a maximalisation, which is a maximal coaction from which the given coaction can be recovered as a quotient. It was shown in \cite{kqcat} that maximalisation can be made into a functor \new{that} has properties which mirror those of normalisation. However, the object map in the functor used in \cite{kqcat} is defined using the axiom of choice, and is not implemented by either of the known constructions of a maximalisation. So there is considerable work to do before we can even begin to forge a boilerplate for maximalisation.

\subsection*{Conventions} We denote the multiplier algebra of a $C^*$-algebra $A$ by $M(A)$, and the group of unitary elements of $M(A)$ by $UM(A)$. By a unitary representation of a locally compact group $G$ in a $C^*$-algebra $A$, we mean a strictly continuous homomorphism $u:G\to UM(A)$. We say that a homomorphism $\phi:A\to M(B)$ of a $C^*$-algebra into a multiplier algebra is \emph{nondegenerate} if $\lsp\{\phi(a)b\}$ is dense in $B$. Every nondegenerate homomorphism $\phi:A\to M(B)$ has a unique extension $\overline\phi:M(A)\to M(B)$ (and we discuss some properties of this extension  in Appendix~\ref{secbarring}). All tensor products of $C^*$-algebras in this paper are spatial.  

Our conventions and notations for categories of $C^*$-algebras and dynamical systems are those of \cite{aHRWsurvey}. These categories are mostly based on those used in \cite{kqcat} and in \cite{kqrproper}, and we explained in \cite{aHRWsurvey} why we are interested in these categories and in crossed-product functors defined on them. 

Throughout, $G$ will be a fixed locally compact group. When $G$ acts on the right of a locally compact space $T$, we write $\rt$ for the induced action of $G$ on $C_0(T)$ defined by \new{$\rt_s(f)(t)=f(t\cdot s)$}.

\section{Crossed products}

Suppose that $\alpha:G\to \Aut A$ is an action of a locally compact group $G$ on a $C^*$-algebra $A$. As in \cite{enchilada} and \cite{tfb2}, we view the crossed product $A\times_\alpha G$ as the completion of $C_c(G,A)$ in the universal norm, so that when we write $A\times_\alpha G$ we have a particular $C^*$-algebra in mind (as opposed to an isomorphism class of $C^*$-algebras as in \cite{raetak}). It is generated by a covariant representation $(i_A,i_G)$ of $(A,\alpha)$ in $M(A\times_\alpha G)$, in the sense that every element $f\in C_c(G,A)$ can be realised as a norm-convergent integral $\int_G i_A(f(s))i_G(s)\,ds$ (see \cite[Corollary~2.36]{tfb2}). The triple $(A\times_\alpha G, i_A,i_G)$ is then universal for covariant representations: for every covariant representation $(\pi,u)$ of $(A,\alpha)$ in a multiplier algebra $M(C)$, there is a unique nondegenerate homomorphism $\pi\times u:A\times_\alpha G\to M(C)$ (called the integrated form) such that $\pi=\overline{\pi\times u}\circ i_A$ and $u=\overline{\pi\times u}\circ i_G$ (see \cite[Proposition~2.39]{tfb2}). In particular (take $A=\C$), the group $C^*$-algebra $C^*(G)$ is generated by a unitary representation $k_G:G\to UM(C^*(G))$, and every strictly continuous unitary representation $u:G\to UM(C)$ gives a nondegenerate representation $\pi_u:C^*(G)\to M(C)$ such that $u=\new{\overline{\pi_u}}\circ k_G$.

We view the reduced crossed product $A\times_{\alpha,r}G$ as the quotient of $A\times_\alpha G$ by the kernel of the regular representation induced by a faithful representation of $A$ (see \cite[Lemma~7.8]{tfb2}, for example). We write $q^r$ for the quotient map,  $(i_A^r,i_G^r)$ for the canonical covariant representation of $(A,\alpha)$ in $M(A\times_{\alpha,r}G)$, and $\pi\times_r u$ for the representation of $A\times_{\alpha,r}G$ coming from a covariant representation $(\pi,u)$ such that $\ker(\pi\times u)\supset\ker q^r$. Similarly, the reduced group algebra $C_r^*(G)$ is generated by $k_G^r:G\to UM(C_r^*(G))$, and a unitary representation $u$ of $G$ such that $\ker \pi_u\supset\ker \pi_\lambda=\ker q^r$ gives a representation $\pi_u^r$ of $C_r^*(G)$.

Our conventions regarding full and reduced  coactions are those of \cite[Appendix A]{enchilada}. A crossed product $A\times_\alpha G$ carries a dual coaction $\hat\alpha$, which is the integrated form of the representation $(i_A\otimes 1,i_G\otimes k_G)$ of $(A,\alpha)$ in $M((A\times_\alpha G)\otimes C^*(G))$. The reduced crossed product $A\times_{\alpha,r}G$, on the other hand, carries both a \emph{full dual coaction}
\[
\hat\alpha^n:=(i_A^r\otimes 1)\times_r(i_G^r\otimes k_G)
\]
(see \cite[Example~A.27]{enchilada}) and a \emph{reduced dual coaction}
\[
\hat\alpha^r:=(i_A^r\otimes 1)\times_r(i_G^r\otimes k_G^r),
\]
which was the dual coaction used in the early papers such as \cite{K} and \cite{LPRS}. The notations $\hat\alpha^n$ and $\hat\alpha^r$ are suggestive: these coactions are, respectively, the normalisation of $\hat\alpha$ \cite[Proposition~A.61]{enchilada} and the reduction of $\hat\alpha$ \cite[Proposition~3.2]{rae}.

Crossed products by coactions can be (and have been) defined in many different  ways. The general upshot is that all these definitions and/or constructions give effectively the same $C^*$-algebra, and indeed that is one of the points which has to be made explicit in the various boilerplates. So we need to be careful to set out which conventions we are using.

Suppose that $\delta :B\to M(B\otimes C^*(G))$ is a full coaction of a locally compact group $G$ on a $C^*$-algebra $B$. A \emph{covariant representation of $(B,\delta)$} in a $C^*$-algebra $C$ is a pair $(\pi,\mu)$ of nondegenerate homomorphisms $\pi:B\to M(C)$ and $\mu:C_0(G)\to M(C)$ such that
\[
\overline{\pi\otimes{\id}}\circ\delta(b)=\Ad\overline{\mu\otimes {\id}}(w_G)(\pi(b)\otimes 1) \quad\text{ in $M(B\otimes C^*(G))$ for every $b\in B$,}
\]
where $w_G$ is the multiplier of $C_0(G)\otimes C^*(G)=C_0(G,C^*(G))$ defined by the function $k_G:G\to M(C_r^*(G))$.
A \emph{crossed product $(A,\rho,\nu)$ for $(B,\delta)$} consists of a covariant representation $(\rho,\nu)$ of $(B,\delta)$ in a $C^*$-algebra $A$ such that the elements $\{\rho(b)\nu(f):b\in B,\,f\in C_0(G)\}$ generate $A$ and every covariant representation $(\pi,\mu):(B,\delta)\to M(C)$ gives a nondegenerate representation $\pi\times\mu:A\to M(C)$ satisfying $\pi=\overline{\pi\times \mu}\circ\rho$ and $\mu=\overline{\pi\times \mu}\circ\nu$. We sometimes call $\pi\times \mu$ the \emph{integrated form} of $(\pi,\mu)$, by analogy with the case of actions. If $(A,\rho,\nu)$ is a crossed product, then we have
\[
A=\clsp\{\rho(b)\nu(f):b\in B\;,f\in C_0(G)\}
\]
(see \cite[Lemma~2.10]{rae} or \cite[Proposition~A.36]{enchilada}). Any two crossed products $(A_1,\rho_1,\nu_1)$ and $(A_2,\rho_2,\nu_2)$ for $(B,\delta)$ are canonically isomorphic: since $(\rho_2,\nu_2)$ is a covariant representation of $(B,\delta)$ in $A_2$, it has an integrated form $\rho_2\times\nu_2:A_1\to M(A_2)$, and this is an isomorphism of $A_1$ onto $A_2$ with inverse $\rho_1\times \nu_1$. Every crossed product $(A,\rho,\nu)$ carries a dual action $\hat\delta:G\to\Aut A$, which is characterised by $\hat\delta_s(\rho(b)\nu(f))=\rho(b)\nu(\rt_s(f))$, and every canonical isomorphism between two crossed products is equivariant for these dual actions.

When we deal with functors, it is helpful to have a specific construction of a crossed product in mind, and we use the one described in \cite[\S A.5]{enchilada}. For any full coaction $(B,\delta)$, the pair $(j_B,j_G):=(\overline{{\id}\otimes\pi_\lambda}\circ\delta, 1\otimes M)$ is a covariant representation of $(B,\delta)$ in $B\otimes \K(L^2(G))$, and it is shown in \cite[Theorem~A.41]{enchilada}, for example, that if $B\times_\delta G$ is the $C^*$-subalgebra of $M(B\otimes \K(L^2(G)))$ generated by the elements $\{j_B(b) j_G(f):b\in B,\,f\in C_0(G)\}$, then the triple $(B\times_\delta G,j_B,j_G)$ is a crossed product for $(B,\delta)$. We call it the \emph{standard crossed product}, and we reserve the notation $(B\times_\delta G,j_B,j_G)$ for this crossed product. We write $j_G^B$ for $j_G$ when there are several systems around.

We use analogous conventions for crossed products by reduced coactions.  A pair $\pi:B\to M(C)$, $\mu:C_0(G)\to M(C)$ of nondegenerate homomorphisms is a covariant representation of a reduced coaction $(B,\delta)$ if
\[
\overline{\pi\otimes{\id}}\circ\delta(b)=\Ad\overline{\mu\otimes {\id}}(w_G^r)(\pi(b)\otimes 1) \quad\text{ in $M(B\otimes C^*_r(G))$ for every $b\in B$,}
\]
where now $w_G^r$ is the multiplier of $C_0(G)\otimes C_r^*(G)=C_0(G,C_r^*(G))$ defined by the function $k_G^r$. A crossed product for $(B,\delta)$ is then a $C^*$-algebra generated by a universal covariant representation, and we write $(B\times_\delta G, j_B, j_G)$ for the \emph{standard crossed product} generated by the covariant representation $(j_B, j_G):=(\overline{{\id}\otimes \pi_\lambda^r}\circ\delta, 1\otimes M)$ in $B\otimes \K(L^2(G))$. (We hope the similar notation for full and reduced does not cause problems: each coaction of whatever sort has exactly one standard crossed product.)

\section{The boilerplate for normal and reduced coactions}\label{nor=red}

A full coaction $(B,G,\delta)$ is \emph{Quigg-normal}, or just \emph{normal}, if the canonical representation $j_B$ of $B$ in $M(B\times_\delta G)$ is faithful. We consider the category $\cscoactnd(G)$ whose objects are full coactions $(B,\delta)$, and in which the morphisms from $(B,\delta)$ to $(C,\epsilon)$ are nondegenerate homomorphisms $\phi:B\to M(C)$ such that $\overline{\phi\otimes \id}\circ\delta=\overline\epsilon\circ\phi$. (This category was  used in \cite{kqcat} and \cite{kqrproper}, and was discussed  in detail in \cite{aHRWsurvey}; see also Proposition~\ref{ex1} below.) In this section, we consider the full subcategory  $\cscoactndn(G)$ of normal coactions, and give a boilerplate for the reduction process on $\cscoactndn(G)$.

The first part of our boilerplate sums up the main properties of reductions, and is taken mostly from  \cite{Q2}.

\begin{prop}\label{propred}
Let $(B,\delta)$ be a normal coaction of a locally compact group $G$.

\smallskip
\textnormal{(a)} The map $\delta^r:=\overline{{\id}\otimes{q^r}}\circ\delta$ is a reduced coaction of $G$ on $B$, called the \emph{reduction of $\delta$}. The map $\Red:(B,\delta)\mapsto (B,\delta^r)$ is a bijection from the class of normal coactions onto the class of reduced coactions. The inverse $\Red^{-1}$ takes a reduced coaction $(B,\epsilon)$ to its \emph{Quiggification} $\epsilon^Q:B\to M(B\otimes C^*(G))$, which is characterised by
\begin{equation}\label{quiggdef}
\overline{\pi\otimes\id}(\epsilon^Q(b))=\Ad\overline{\mu\otimes\id}(w_G)(\pi(b)\otimes 1)
\end{equation}
for every covariant representation $(\pi,\mu)$ of $(B,\epsilon)$.

\smallskip
\textnormal{(b)} Suppose that $\pi:B\to M(C)$, $\mu:C_0(G)\to M(C)$ are nondegenerate representations. Then $(\pi,\mu)$ is a covariant representation of $(B,\delta)$ if and only $(\pi,\mu)$ is a covariant representation of $(B,\delta^r)$.

\smallskip
\textnormal{(c)} A triple $(A,\rho,\nu)$ is a crossed product for $(B,\delta)$ if and only if it is a crossed product for $(B,\delta^r)$. The standard crossed products $B\times_\delta G$ and $B\times_{\delta^r}G$ are the same subalgebra of $M(B\otimes \K(L^2(G)))$.
\end{prop}

\begin{remark} If $(B,\epsilon)$ is a reduced coaction of $G$, then applying (b) and (c) with $\delta:=\epsilon^Q$ gives analogues of (b) and (c) relating $(B,\epsilon^Q)$ to $(B,\epsilon)$.
\end{remark}

\begin{proof}[Proof of Proposition~\ref{propred}]
That $\delta^r$ is a reduced coaction is proved in Proposition~3.3 of \cite{Q2}. It is shown in the proof of \cite[Theorem~4.7]{Q2} that if $(B,\epsilon)$ is a reduced coaction, then there is a unique coaction $\epsilon^Q$ on $B$ satisfying
\begin{equation}\label{quiggdef2}
\overline{j_B^r\otimes\id}(\epsilon^Q(b))=\Ad\overline{j_G^r\otimes\id}(w_G)(j_B^r(b)\otimes 1),
\end{equation} that $\epsilon^Q$ is normal, and that $(\epsilon^Q)^r=\epsilon$. This immediately gives \eqref{quiggdef} for $(\pi,\mu)=(j_B^r,j_G^r)$, and for general $(\pi,\mu)$, we apply $\overline{\pi\times_r\mu}\otimes{\id}$ to both sides of \eqref{quiggdef2} and use \eqref{N1}. This implies that $\Red$ is onto, and the uniqueness assertion in \cite[Theorem~4.7]{Q2} implies that $\Red$ is one-to-one. This gives (a).

Part (b) is proved in Proposition~3.7 of \cite{Q2}, but we can also  deduce it from (a). If $(\pi,\mu)$ is covariant for $(B,\delta)$, then we have
\begin{align}\label{fulltored}
\overline{{\id}\otimes q^r}\circ \overline{\pi\otimes{\id}}(\delta(b))
&=\overline{{\id}\otimes q^r}\big(\Ad\overline{\mu\otimes{\id}}(w_G)(\pi(b)\otimes 1)\big)\notag\\
&=\Ad\big(\overline{{\id}\otimes q^r}\circ\overline{\mu\otimes{\id}}(w_G)\big)(\pi(b)\otimes 1).
\end{align}
Using \eqref{N2} to pull ${\id}\otimes q^r$ past $\pi\otimes {\id}$   shows that the left of \eqref{fulltored} is $\overline{\pi\otimes{\id}}(\delta^r(b))$, and a calculation in $M(C_0(G,C^*_r(G)))$ shows that $\overline{{\id}\otimes q^r}(w_G)=w_G^r$, so \eqref{fulltored} implies that $(\pi,\mu)$ is covariant for $(B,\delta^r)$.

Now suppose that $(\pi,\mu)$ is covariant for $(B,\delta^r)$, and recall from (a) that $\delta=(\delta^r)^Q$. Thus we can deduce from \eqref{quiggdef} (with $\epsilon=\delta^r$) that
\begin{align*}
\overline{\pi\otimes{\id}}(\delta(b))
=\overline{\pi\otimes{\id}}((\delta^r)^Q(b))
=\Ad\overline{\mu\otimes\id}(w_G)(\pi(b)\otimes 1),
\end{align*}
which is the required covariance.

Suppose that $(A,\rho,\nu)$ is a crossed product for $(B,\delta)$. Part (b) implies that $(\rho,\nu)$ is covariant for $(B,\delta^r)$. If $(\pi,\mu)$ is a covariant representation for $(B,\delta^r)$, then the other half of part (b) implies that $(\pi,\mu)$ is covariant for $(B,\delta)$, and hence there is a representation $\pi\times \mu$ of $A$ such that $\pi=\overline{\pi\times \mu}\circ \rho$ and $\mu=\overline{\pi\times \mu}\circ \nu$. Thus $(A,\rho,\nu)$ is a crossed product for $(B,\delta^r)$. A similar argument gives the other direction. For the last comment in (c), we note that, using \eqref{barcomp},
\[
\overline{{\id}\otimes\pi_\lambda^r}\circ\delta^r=\overline{{\id}\otimes\pi_\lambda^r}\circ\overline{{\id}\otimes q^r}\circ\delta=\overline{\id\otimes(\new{\overline{\pi_\lambda^r}}\circ q^r)}\circ\delta=\overline{{\id}\otimes\pi_\lambda}\circ\delta,
\]
so the two crossed products have exactly the same generators.
\end{proof}

We now want to form a category $\cscoactndr(G)$ of reduced coactions so that $\Red$ is a functor, and then establish that $\Red$ is an isomorphism of categories. As before, the morphisms from one reduced coaction $(B,\delta)$ to another $(C,\epsilon)$ will be nondegenerate homomorphisms $\phi:B\to M(C)$ such that $\overline{\phi\otimes \id}\circ\delta=\overline\epsilon\circ\phi$, though this time the equation holds in $M(C\otimes C_r^*(G))$ rather than $M(C\otimes C^*(G))$. The composition of morphisms in $\cscoactndr(G)$ will be defined as in $\cscatnd$: $\psi\circ\phi$ is by definition the usual composition $\overline\psi\circ \phi$. There are some things to check before we can confidently assert that this gives a category; the key to doing this is the following lemma, which is also the main step in proving that $\Red$ is an isomorphism. It was proved by a different argument in \cite[Appendix~A]{kqrproper}.

\begin{lemma}\label{cov=cov_r}
Suppose that $(B,\delta)$ and $(C,\epsilon)$ are objects in $\cscoactndn(G)$, and $\phi:B\to M(C)$ is a nondegenerate homomorphism. Then
\[
\overline{\phi\otimes\id}\circ\delta=\overline \epsilon\circ\phi\Longleftrightarrow
\overline{\phi\otimes\id}\circ\delta^r=\new{\overline{\epsilon^r}}\circ\phi.
\]
\end{lemma}

\begin{proof}
The forward implication follows from another application of \eqref{N2}. So  suppose that $\overline{\phi\otimes\id}\circ\delta^r=\new{\overline{\epsilon^r}}\circ\phi$.  We know from Proposition~\ref{propred}(c) that $(j_C,j_G^C)$ is a covariant representation of $(C,\epsilon^r)$, and a calculation using \eqref{N1} shows that $(\new{\overline{j_C}}\circ\phi,j_G^C)$ is a covariant representation of $(B,\delta^r)$. Now Proposition~\ref{propred}(b) implies that $(\new{\overline{j_C}}\circ\phi,j_G^C)$ is a covariant representations of $(B,\delta)$, and hence
\begin{align}
\overline{j_C\otimes\id}\big(\overline{\phi\otimes\id}\circ\delta(b)\big)
&=\overline{(\new{\overline{j_C}}\circ \phi)\otimes\id}\circ\delta(b)\quad \text{(using \eqref{N1}\,)}\notag\\
&=\Ad\overline{j_G^C\otimes\id}(w_G)(\new{\overline{j_C}}(\phi(b))\otimes 1).\label{extcov}
\end{align}
The map $m\mapsto \Ad\overline{j_G^C\otimes\id}(w_G)(\new{\overline{j_C}}(m)\otimes 1)$ is the strictly continuous extension of
\[
c\mapsto \Ad\overline{j_G^C\otimes\id}(w_G)(j_C(c)\otimes 1)=\overline{j_C\otimes{\id}}\circ \epsilon,
\]
and hence coincides with $\overline{j_C\otimes{\id}}\circ \overline\epsilon$. Thus \eqref{extcov} gives
\[
\overline{j_C\otimes\id}\big(\overline{\phi\otimes\id}\circ\delta(b)\big)=\overline{j_C\otimes{\id}}\circ \overline\epsilon(\phi(b)),
\]
which since $\epsilon$ is normal implies that $\overline{\phi\otimes\id}\circ\delta=\overline\epsilon\circ\phi$.
\end{proof}

Using Lemma~\ref{cov=cov_r} and the corresponding properties of morphisms in $\cscoactndn(G)$, we can see that the composition of two morphisms in $\cscoactndr(G)$ is a morphism in $\cscoactndr(G)$. The other necessary properties of composition follow from their counterparts in $\cscatnd$. Lemma~\ref{cov=cov_r} also tells us that $\Red(\phi):=\phi$ maps morphisms in $\cscoactndn(G)$ to morphisms in $\cscoactndr(G)$, and because composition is defined the same way in both categories, $\Red$ is a functor.

\begin{thm}\label{bashatbplate} Suppose that $G$ is a locally compact group.

\smallskip
\textnormal{(a)} The functor $\Red:\cscoactndn(G)\to \cscoactndr(G)$ such that
\[
\Red(B,\delta):=(B,\delta^r):=(B,\overline{\id\otimes\, {q^r}}\circ\delta)\quad\text{and}\quad \Red(\phi):=\phi
\]
is an isomorphism of categories. The inverse $\Red^{-1}$ takes a reduced coaction $(B,\delta)$ to the normal coaction $(B,\quigg)$ characterised by \eqref{quiggdef}, and $\Red^{-1}(\phi)=\phi$.

\smallskip
\textnormal{(b)} The functor $\Red$ takes a full dual system $(A\times_{\alpha, r}G,\hat\alpha^n)$ to the reduced dual system $(A\times_{\alpha, r}G,\hat\alpha^r)$, and the inverse $\Red^{-1}$ takes the reduced dual system back to the full dual system.
\end{thm}

\begin{proof}
Proposition~\ref{propred} says that $\Red$ is a bijection on objects, with inverse as described. Now Lemma~\ref{cov=cov_r} says that $\Red$ is surjective and injective on morphisms. This gives (a).

For (b), we compute the reduction using Lemma~\ref{barintform}:
\begin{align*}
(\hat\alpha^n)^r=\big((i_A^r\otimes 1)\times_r(i_G^r\otimes k_G)\big)^r
&=\overline{\id\otimes\, {q^r}}\circ\big((i_A^r\otimes 1)\times_r(i_G^r\otimes k_G)\big)\\
&=\big(\overline{\id\otimes\, {q^r}}\circ(i_A^r\otimes 1)\big)\times_r\big(\overline{\id\otimes\, {q^r}}\circ(i_G^r\otimes k_G)\big)\\
&=(i_A^r\otimes 1)\times_r(i_G^r\otimes k_G^r)=\hat\alpha^r.
\end{align*}
That $\Red^{-1}$ does what is claimed is just the elementary property of inverses.
\end{proof}

\section{The boilerplate for normalising}\label{secnormalising}

We begin by summing up the properties of the normalisation of a coaction $(B,\delta)$. We prefer to view the normalisation as a coaction on a quotient $B^n$ of $B$. This differs from the original description in \cite{Q2}, where the normalisation acts on the canonical image $j_B(B)$ of $B$ in $M(B\times_\delta G)$; it also differs from the more recent discussion in \cite{kqcat}, where a normalisation is a normal system with a distinguished morphism from $(B,\delta)$ \new{that} induces an isomorphism on crossed products.

\begin{prop}\label{propnormal}
Suppose that $(B,\delta)$ is a full coaction of a locally compact group $G$. Let $q^n$ be the quotient map from $B$ to $B^n:=B/\ker j_B$, and let $j_B^n:B^n\to M(B\times_\delta G)$ be the injection such that $j_B=j_B^n\circ q^n$.

\smallskip
\textnormal{(a)} There is a unique homomorphism $\delta^n:B^n\to M(B^n\otimes C^*(G))$ such that 
\begin{equation}\label{eq-q^n}\delta^n\circ q^n=\overline{q^n\otimes\id}\circ\delta,
\end{equation} 
and $\delta^n$ is a normal coaction of $G$ on $B^n$, called the \emph{normalisation of $\delta$}.

\smallskip
\textnormal{(b)} If $(\pi,\mu)$ is a covariant representation of $(B^n,\delta^n)$, then $(\pi\circ q^n,\mu)$ is a covariant representation of $(B,\delta)$. Conversely, if $(\rho,\mu)$ is a covariant representation of $(B,\delta)$, then $\ker\rho\supset\ker j_B$, and if $\rho^n$ is the representation of $B^n$ such that $\rho=\rho^n\circ q^n$, then $(\rho^n,\mu)$ is a covariant representation of $(B^n,\delta^n)$.

\smallskip
\textnormal{(c)}  The triple $(B^n\times_{\delta^n}G,j_{B^n}\circ q^n,j_G)$ is a crossed product for $(B,\delta)$, and the triple $(B\times_\delta G,j_B^n,j_G)$ is a crossed product for $(B^n,\delta^n)$.

\smallskip
\textnormal{(d)} Suppose that $N$ is a closed normal subgroup of $G$, that $q:G\to G/N$ is the quotient map, and that $\pi_q$ is the integrated form of $k_{G/N}\circ q$. Then $\delta|:=\overline{{\id}\otimes \pi_q}\circ \delta$ is a coaction of $G/N$ on $B$, called the \emph{restriction of $\delta$ to $G/N$}. Write $j_G|$ for the restriction of $\new{\overline{j_G}}:C_b(G)\to M(B\times_\delta G)$ to $C_0(G/N)$. Then $(j_B,j_G|)$ is a covariant representation of $(B,\delta|)$, and if $j_B\times j_G|:B\times_{\delta|}(G/N)\to M(B\times_\delta G)$ is injective, then normalisation commutes with restriction, in the sense that $(\delta|)^n=(\delta^n)|$.
\end{prop}

\begin{remark} If we start with a normal coaction $(B,\delta)$, then $j_B$ is injective, $B/\ker j_B=B$ and $(B^n,\delta^n)=(B,\delta)$, so that $(B,\delta)$ is its own normalisation.
\end{remark}

\begin{proof}
Quigg proved in  \cite[Propositions 2.3 and~2.5]{Q2} that the formula
\begin{equation*}
\delta^{j_G}(j_B(b)):=\Ad\overline{j_G\otimes\id}(w_G)(j_B(b)\otimes 1)
\end{equation*}
defines a normal coaction $\delta^{j_G}$ on $j_B(B)$. Since $j_B^n$ is an isomorphism of $B^n$ onto $j_B(B)$, we can pull $\delta^{j_G}$ back to a normal coaction $\delta^n$ on $B^n$ such that
\begin{equation}\label{defnormal}
\overline{j_B^n\otimes\id}(\delta^n(q^n(b)))=\Ad\overline{j_G\otimes\id}(w_G)(j_B(b)\otimes 1).
\end{equation}
Now we compute, using the covariance of $(j_B,j_G)$ and \eqref{N1}:
\begin{align*}
\overline{j_B^n\otimes\id}(\delta^n(q^n(b)))&=\Ad\overline{j_G\otimes\id}(w_G)(j_B(b)\otimes 1)
=\overline{j_B\otimes\id}(\delta(b))\\
&=\overline{(\new{\overline {j_B^n}}\circ q^n)\otimes\id}(\delta(b))
=\overline{j_B^n\otimes\id}\circ\overline{q^n\otimes\id}(\delta(b));
\end{align*}
since $j_B^n\otimes\id$ is injective on $B^n\otimes C^*(G)$, this implies that $\delta^n\circ q^n=\overline{q^n\otimes\id}\circ\delta$. The uniqueness follows from the surjectivity of $q^n$, and we already know that $\delta^n$ is a normal coaction, so we have proved (a).

For (b), suppose first that $(\pi,\mu)$ is a covariant representation of $(B^n,\delta^n)$. Then the calculation
\begin{align*}
\overline{(\pi\circ q^n)\otimes\id}\circ \delta(b)
&=\overline{\pi\otimes \id}\circ\overline{q^n\otimes \id}\circ\delta(b)\\
&=\overline{\pi\otimes \id}\circ\delta^n\circ q^n(b)\quad\text{(using (a)\,)}\\
&=\Ad\overline{\mu\otimes\id}(w_G)(\pi\circ q^n(b)\otimes 1)
\end{align*}
shows that $(\pi\circ q^n,\mu)$ is covariant for $(B,\delta)$. Now suppose that $(\rho,\mu)$ is a covariant representation of $(B,\delta)$. Then the identity $\rho=\overline{\rho\times\mu}\circ j_B$ implies that $\ker\rho\supset \ker j_B$, so we can factor $\rho=\rho^n\circ q^n$ and compute:
\begin{align*}
\overline{\rho^n\otimes\id}\circ \delta^n(q^n(b))
&=\overline{\rho^n\otimes \id}\circ\overline{q^n\otimes\id}\circ\delta(b)\quad\text{(using (a)\,)}\\
&=\overline{(\rho^n\circ q^n)\otimes \id}\circ\delta(b)=\overline{\rho\otimes \id}\circ\delta(b)\\
&=\Ad\overline{\mu\otimes\id}(w_G)(\rho(b)\otimes 1)\\
&=\Ad\overline{\mu\otimes\id}(w_G)(\rho^n(q^n(b))\otimes 1),
\end{align*}
which since $q^n$ is surjective implies that $(\rho^n,\mu)$ is covariant for $(B^n,\delta^n)$.

The first assertion in part (c) is proved in \cite[Proposition~2.6]{Q2}. For the second, suppose that $(\pi,\mu)$ is a covariant representation of $(B^n,\delta^n)$. Part (b) implies, first, that $(j_B^n,\mu)$ is covariant for $(B^n,\delta^n)$, and, second,  that $(\pi\circ q^n,\mu)$ is covariant for $(B,\delta)$. Then  $\overline{(\pi\circ q^n)\times \mu}\circ j_G=\mu$ and $\overline{(\pi\circ q^n)\times \mu}\circ j_B=\pi\circ q^n$, which gives
\[
\overline{(\pi\circ q^n)\times \mu}\circ j_B^n(q^n(b))=\overline{(\pi\circ q^n)\times \mu}\circ j_B=\pi(q^n(b)).
\]
Thus $(\pi\circ q^n)\times \mu$ has the properties required of $\pi\times \mu$.

The assertion that $\delta|$ is a coaction is discussed in \cite[Example~A.28]{enchilada}, and it follows from Lemma~\ref{covrestricts} that $(j_B,j_G|)$ is covariant for $(B,G/N,\delta|)$.

Now we consider the standard crossed product $(B\times_{\delta|} (G/N), j_B^{G/N}, j_{G/N})$ for $(B,\delta|)$, and let $q^n_{G/N}:B\to B/\ker{j_B^{G/N}}$ be the quotient map.  By Proposition~\ref{propnormal}(a), $(\delta|)^n$ is the unique homomorphism such that
\begin{equation}
\label{!normalisedrestriction}
(\delta|)^n\circ q^n_{G/N}=\overline{q^n_{G/N}\otimes\id}\circ\delta|.
\end{equation}
By assumption, $j_B\times j_G|$ is an isomorphism of $B\times_{\delta|}(G/N)$ onto its range
\[B\times_{\delta, r}(G/N):=\clsp\{j_B(b)\overline{j_G}|(f):b\in B, f\in C_0(G/N)\}\subset M(B\times_\delta G).\]
(This is the crossed product of $B$ by the homogeneous space $G/N$ discussed in \S\ref{subsecMansfield} below). So $(B\times_{\delta, r}(G/N), j_B, j_G|)$ is a crossed product for $(B,\delta|)$, and hence $\ker{j_B^{G/N}}=\ker j_B$ and $q^n_{G/N}=q^n:B\to B/\ker j_B$. Now
\begin{align*}
(\delta^n)|\circ q^n_{G/N}
&=
(\delta^n)|\circ q^n=
\big(\overline{{\id}\otimes \pi_q}\circ\delta^n\big)\circ q^n\\
&=
\overline{{\id}\otimes \pi_q}\circ\big(\delta^n\circ q^n\big)
=\overline{{\id}\otimes \pi_q}\circ\overline{q^n\otimes \id}\circ\delta\\
&=\overline{q^n\otimes \id}\circ\overline{{\id}\otimes \pi_q}\circ\delta
=\overline{q^n\otimes\id}\circ\delta|\\
&=\overline{q_{G/N}^n\otimes\id}\circ\delta| \quad\text{(since $q^n_{G/N}=q^n$)},
\end{align*}
and comparing with \eqref{!normalisedrestriction} gives $(\delta|)^n=(\delta^n)|$.
\end{proof}

We now state our result on the categorical properties of normalisation, which is taken mainly from \cite{kqcat}.

\begin{thm}\label{bplatenor} Suppose that $G$ is a locally compact group.

\smallskip
\textnormal{(a)} If $\phi:(B,\delta)\to (C,\epsilon)$ is a morphism in $\cscoactnd(G)$ there is a unique morphism $\phi^n:(B^n,\delta^n)\to (C^n,\epsilon^n)$ such that $\phi^n\circ q^n=\new{\overline{q^n}}\circ\phi$, and then the assignments
\[
(B,\delta)\mapsto (B^n,\delta^n)\quad\text{ and}\quad\phi\mapsto\phi^n
\]
constitute a functor $\Nor:\cscoactnd(G)\to \cscoactndn(G)$.

\smallskip
\textnormal{(b)} Let $\Inc:\cscoactndn(G)\to \cscoactnd(G)$ be the inclusion functor. Then $\Nor\circ\Inc$ is the identity functor $\Id$ on $\cscoactndn(G)$. If $(B,\delta)$ is an object in $\cscoactnd(G)$ and $(C,\epsilon)$ is an object in $\cscoactndn(G)$, then the map $\phi\mapsto\phi^n$ is a bijection from $\Mor((B,\delta),(C,\epsilon))$ to $\Mor((B^n,\delta^n),(C,\epsilon))$, and these bijections are natural in both variables.

\smallskip
\textnormal{(c)} Every full dual system $(A\times_{\alpha,r}G,\hat\alpha^n)$  is normal, and is the normalisation of the dual system $(A\times_{\alpha}G,\hat\alpha)$.
\end{thm}

As pointed out in \cite{kqcat}, everything hinges on a universal property of normalising maps. We'll begin by looking carefully at this universal property. The next lemma is similar to Lemma~2.1 of \cite{ekq}, and was stated in \cite[item (ii), page~430]{kqcat}.

\begin{lemma}\label{uniquepropnormal}
Suppose that $(B,G,\delta)$ is a full coaction.   Then the quotient map $q^n:B\to B^n=B/\ker j_B$ is a morphism in $\cscoactnd(G)$ from $(B,\delta)$ to $(B^n,\delta^n)$, and if $\psi:(B,\delta)\to (C,\epsilon)$ is a morphism to a normal coaction, then there is a unique morphism $\theta:(B^n,\delta^n)\to (C,\epsilon)$ such that $\psi=\theta\circ q^n$.
\end{lemma}

\begin{proof} Equation~\eqref{eq-q^n} says that $q^n$ is a morphism.
Since the equation $\psi=\theta\circ q^n$ determines the value of $\theta(q^n(b))$, there is at most one homomorphism $\theta$ with the required property. We need to show that $\ker \psi$ contains $\ker q^n=\ker j_B$. Since $(C,\epsilon)$ is normal, $\ker\psi=\ker \new{\overline{j_C}}\circ\psi$. Now we compute, using Proposition~\ref{apponbarring} and the equivariance of $\psi$:
\begin{align*}
\overline{(\new{\overline{j_C}}\circ\psi)\otimes\id}\circ\delta(b)
&=\overline{j_C\otimes\id}\circ\overline{\psi\otimes\id}\circ\delta(b)\\
&=\overline{j_C\otimes\id}\circ\overline\epsilon\circ\psi(b)\\
&=\overline{\overline{j_C\otimes\id}\circ\epsilon}\circ\psi(b)\\
&=\Ad\overline{j_G\otimes\id}(w_G)(\new{\overline{j_C}}(\psi(b))\otimes 1),
\end{align*}
where at the last step we again used strict continuity to extend the covariance relation to $M(C)$.
Thus $(\new{\overline{j_C}}\circ\psi,j_G^C)$ is a covariant representation of $(B,\delta)$ in $M(C\times_\delta G)$, and in particular $\new{\overline{j_C}}\circ\psi$ factors through $j_B$. Thus $\ker\psi=\ker\new{\overline{j_C}}\supset \ker j_B$, and there is a homomorphism $\theta:B^n\to M(C)$ such that $\psi=\theta\circ q^n$.

To see that $\theta$ is a morphism in $\cscoactnd(G)$, we use that $\psi$ and $q^n$ are:
\begin{align*}
\overline\epsilon(\theta(q^n(b)))&=\overline\epsilon(\psi(b))=\overline{\psi\otimes\id}(\delta(b))\\
&=\overline{(\theta\circ q^n)\otimes\id}(\delta(b))=\overline{\theta\otimes\id}\circ\overline{q^n\otimes\id}(\delta(b))\\
&=\overline{\theta\otimes\id}\circ\delta^n(q^n(b)).\qedhere
\end{align*}
\end{proof}

\begin{proof}[Proof of Theorem~\ref{bplatenor}]
Applying  Lemma~\ref{uniquepropnormal} to $\psi=\new{\overline{q^n}}\circ\phi$ shows that there is  such a morphism $\phi^n$, and it follows from the uniqueness in Lemma~\ref{uniquepropnormal} that $\phi\mapsto\phi^n$ is functorial. This gives (a).
The first assertion in (b) amounts to the observations that normalisation leaves a normal system unchanged, and that if $\phi$ is a morphism between normal systems, then $\phi^n=\phi$. Since the quotient map $q^n$ for the normal system $(C,\epsilon)$ is the identity, Lemma~\ref{uniquepropnormal} says that $\phi\mapsto \phi^n$ is a bijection with inverse $\psi\mapsto \psi\circ q^n$. The functoriality of $\phi\mapsto\phi^n$ implies that the bijections are natural in $(B,\delta)$, and straightforward applications of the uniqueness in Lemma~\ref{uniquepropnormal} show that they are natural in $(C,\epsilon)$.

Since normalisations are normal, it suffices to prove the second assertion in (c), and this is Proposition~A.61 of \cite{enchilada}. The crucial ingredient in the proof of that proposition is Theorem~4.1 of \cite{rae}, which says that if $\pi\times U$ is a faithful representation of $A\times_\alpha G$ on $\HH$, then the regular representation
\[
\big((\pi\times U)\otimes\lambda)\circ\hat\alpha, 1\otimes M\big))=\big((\pi\otimes 1)\times(U\otimes\lambda),1\otimes M\big)
\]
gives a faithful representation of $(A\times_\alpha G)\times_{\hat\alpha}G$ on $\HH\otimes L^2(G)=L^2(G,\HH)$. This implies that $\ker j_{A\times_\alpha G}=\ker \big((\pi\otimes 1)\times(U\otimes\lambda)\big)$, and since the operator $W$ defined by $(W\xi)(s)=U_s^*(\xi(s)$ intertwines $(\pi\otimes 1)\times(U\otimes\lambda)$ and the regular representation $\widetilde\pi\times(1\otimes \lambda)$ induced from the faithful representation $\pi$ of $A$, we deduce that $(A\times_\alpha G)^n=(A\times_\alpha G)/\ker j_{A\times_\alpha G}$ is the reduced crossed product. This implies $(c)$.
\end{proof}

Our definition of normalisation is, as we pointed out at the start of the section, not quite the same as the one used in \cite{Q2}: we replace the $C^*$-algebra $j_B(B)$ by the isomorphic algebra $B^n=B/\ker j_B$. On the other hand, it looks nothing like the concept of normalisation used in \cite{kqcat} (see \cite[page~423]{kqcat}). The connection is the universal property described in Lemma~\ref{uniquepropnormal}, which says that the pair $((B^n,\delta^n),q^n)$ is an initial object in the comma category $(B,\delta)\downarrow \cscoactndn(G)$. In fact, these initial objects are the normalisations studied in \cite{kqcat}:

\begin{prop}\label{norvsKQnor}
Suppose that $(B,\delta)$ is a full coaction, $(D,\eta)$ is a normal coaction, and $\phi:(B,\delta)\to (D,\eta)$ is a morphism. Suppose that for every morphism $\psi:(B,\delta)\to (C,\epsilon)$ to a normal coaction, there is a unique morphism $\theta:(D,\eta)\to (C,\epsilon)$ such that $\psi=\theta\circ \phi$. Then $\phi(B)=D$, and the map $\phi\times G$ (described in Proposition~\ref{defCPfunctor} below) is an isomorphism of $B\times_\delta G$ onto $D\times_\eta G$.
\end{prop}

\begin{proof}
We know from Lemma~\ref{uniquepropnormal} that there is a morphism $\theta:(B^n,\delta^n)\to (D,\eta)$ such that $\theta\circ q^n=\phi$, and by hypothesis that there is a unique morphism $\tau:(D,\eta)\to (B^n,\delta^n)$ such that $\tau\circ \phi=q^n$. Then the uniqueness in Lemma~\ref{uniquepropnormal} implies that $\tau\circ\theta$ is the identity on $B^n$, and the uniqueness in the hypothesis that $\theta\circ \tau$ is the identity on $D$. Since the isomorphisms in $\cscatnd$ are the isomorphisms of $C^*$-algebras (see Proposition~1 of \cite{aHRWsurvey}, for example), we deduce that $\theta$ is an isomorphism, and is in particular surjective. Now the equation $\theta\circ q^n=\phi$ and the surjectivity of $q^n$ imply that $\phi(B)=D$.

Since $\phi(B)=D$, every generator of $D\times_\eta G$ is in the range of $\phi\times G$, and $\phi\times G$ is surjective. To see that $\phi\times G$ is injective, we let $(\pi,\mu)$ be a covariant representation of $(B,\delta)$, and prove that there is a representation $\rho$ of $D\times_\eta G$ such that $\pi\times\mu=\rho\circ(\phi\times G)$.

Recall from \cite[Proposition~2.3]{Q2} that the system $(\pi(B),\delta^\mu)$ (in which $\delta^\mu$ is conjugation by $\overline{\mu\otimes{\id}}(w_G)$) is normal. The covariance of $(\pi,\mu)$ says that $\pi:B\to \pi(B)$ is $\delta$--$\delta^\mu$ equivariant, and hence is a morphism from $(B,\delta)$ to $(\pi(B),\delta^\mu)$. Since $((D,\eta),\phi)$ is an initial object, there is a morphism $\tau:(D,\eta)\to (\pi(B),\delta^\mu)$ such that $\pi=\tau\circ \phi$. Then $(\tau,\mu)$ is a covariant representation of $(D,\eta)$:
\begin{align*}
\overline{\tau\otimes{\id}}\circ\eta(\phi(b))
&=\overline{\tau\otimes{\id}}\circ\overline{\phi\otimes{\id}}\circ\delta(b)\\
&=\overline{(\tau\circ \phi)\otimes{\id}}\circ\delta(b)\quad\text{(using \eqref{N1}\,)}\\
&=\overline{\pi\otimes{\id}}\circ\delta(b)\\
&=\Ad\mu\otimes{\id}(w_G)(\pi(b)\otimes 1)\\
&=\Ad\mu\otimes{\id}(w_G)(\tau(\phi(b))\otimes 1).
\end{align*}
Now $\rho:=\tau\times\mu$ satisfies
\[
\overline{\rho\circ(\phi\times G)}\circ j_B=\overline{\tau\times\mu}\circ\overline{\phi\times G}\circ j_B=\overline{\tau\times\mu}\circ j_D\circ \phi=\tau\circ \phi=\pi,
\]
and (similarly) $\overline{\rho\circ(\phi\times G)}\circ j_G=\mu$, so $\rho$ has the required property.
\end{proof}

\begin{cor}\label{natisonor}
For every full coaction $(B,\delta)$, the map $q^n\times G$ is an isomorphism of $B\times_\delta G$ onto $B^n\times_{\delta^n}G$.
\end{cor}

\begin{proof}
Lemma~\ref{uniquepropnormal} says that we can apply Proposition~\ref{norvsKQnor} to $q^n$.
\end{proof}

\section{Applications.}\label{secappl}

\subsection{Crossed-product functors}\label{secCP}
We now want to define crossed-product functors on our various categories of coactions. With our conventions, there is an obvious way to define $\CP$ on objects: we send a coaction $(B,\delta)$ to the standard crossed product $B\times_\delta G\subset M(B\otimes\K(L^2(G)))$. So the problem is to decide what to do with morphisms. The next result tells us how to do this for full coactions.

\begin{prop}\label{defCPfunctor}
Suppose that $\phi:(B,\delta)\to (C,\epsilon)$ is a morphism in $\cscoactnd(G)$. Then there is a unique nondegenerate homomorphism $\phi\times G:B\times_\delta G\to M(C\times_\epsilon G)$ such that $\overline{\phi\times G}\circ j_B=\new{\overline{j_C}}\circ\phi$ and $\overline{\phi\times G}\circ j_G^B= j_G^C$. The assignments $(B,\delta)\mapsto B\times_\delta G$ and $\phi\mapsto \phi\times G$ define a functor $\CP:\cscoactnd(G)\to \cscatnd$.
\end{prop}

\begin{proof}
The existence of $\phi\times G:=(\overline\phi\circ j_B)\times j_G^C$ is proved in \cite[Lemma~A.46]{enchilada}, and it is unique because the elements $j_B(b)j_G(f)$ span a dense subspace of $B\times_\delta G$.

Suppose that $\phi:(B,\delta)\to (C,\epsilon)$ and $\psi:(C,\epsilon)\to (D,\eta)$ are morphisms in $\cscoactnd(G)$. Then the composition in \new{$\cscatnd$} of $\psi\times G$ and $\phi\times G$ is the nondegenerate homomorphism $\overline{\psi\times G}\circ(\phi\times G)$, which satisfies
\begin{align*}
\new{\overline{\overline{\psi\times G}\circ(\phi\times G)}}\circ j_B
&=\overline{\psi\times G}\circ\overline{\phi\times G}\circ j_B
=\overline{\psi\times G}\circ\new{\overline{j_C}}\circ\phi\\
&=\overline{\overline{\psi\times G}\circ j_C}\circ\phi
=\overline{\new{\overline{j_D}}\circ\psi}\circ \phi=\new{\overline{j_D}}\circ(\overline\psi\circ\phi),
\end{align*}
and similarly $\new{\overline{\overline{\psi\times G}\circ(\phi\times G)}}\circ j_G^B=j_G^D$. Thus $\overline{\psi\times G}\circ(\phi\times G)$ has the property \new{that} characterises $(\overline\psi\circ\phi)\times G$, and uniqueness gives
\[
\overline{\psi\times G}\circ(\phi\times G)=(\overline\psi\circ\phi)\times G.
\]
The other necessary properties of composition follow from their counterparts in $\cscatnd$.  So $\CP:\cscoactnd(G)\to \cscatnd$ is a functor.
\end{proof}

We define $\CPn$ to be the restriction of $\CP$ to the subcategory $\cscoactndn(G)$. We can then use our boilerplate and $\CP^n$ to get a crossed-product functor on the category $\cscoactndr(G)$ of reduced coactions defined preceding Theorem~\ref{bashatbplate}. (The existence of this functor was asserted without proof in Theorem~4.2 and in the proof of \new{Proposition}~6.1 in \cite{kqrproper}.)

\begin{cor}
Suppose that $\phi:(B,\delta)\to (C,\epsilon)$ is a morphism in $\cscoactndr(G)$. Then there is a unique nondegenerate homomorphism $\phi\times G:B\times_\delta G\to M(C\times_\epsilon G)$ such that $\overline{\phi\times G}\circ j_B=\new{\overline{j_C}}\circ\phi$ and $\overline{\phi\times G}\circ j_G^B= j_G^C$. The assignments $(B,\delta)\mapsto B\times_\delta G$ and $\phi\mapsto \phi\times G$ define a functor $\CP^r:\cscoactndr(G)\to \cscatnd$.
\end{cor}

\begin{proof}
The functor $\CP^r:=\CP^n\circ\Red^{-1}$ has the required properties.
\end{proof}

The way we have defined $\CP^r$ it trivially satisfies $\CP^n=\CP^r\circ\Red$. The analogous statement for normalisations is a little more subtle.

\begin{prop}
The maps
\[\{q^n\times G:B\times_\delta G\to B^n\times_{\delta^n} G:(B,\delta)\in \cscoactnd(G)\}\]
implement a natural isomorphism between $\CP$ and $\CP^n\circ \Nor$.
\end{prop}

\begin{proof}
We know from Corollary~\ref{natisonor} that $q^n\times G$ is an isomorphism. To see that the isomorphism is natural we need to check that if $\phi$ is a morphism then
\[
\overline{q^n\times G}\circ (\phi\times G)=\overline{\phi^n\times G}\circ (q^n\times G).
\]
This is straightforward: for example,
\begin{align*}
\overline{\overline{q^n\times G}\circ (\phi\times G)}\circ j_B
&=\new{\overline{j_{C^n}}}\circ(q^n\circ\phi)=\new{\overline{j_{C^n}}}\circ(\phi^n\circ q^n)\\
&=\overline{\overline{\phi^n\times G}\circ (q^n\times G)}\circ j_B,
\end{align*}
and we similarly have
\[
\overline{\overline{q^n\times G}\circ (\phi\times G)}\circ j_G^B=\overline{\overline{\phi^n\times G}\circ (q^n\times G)}\circ j_G^B.\qedhere
\]
\end{proof}

\subsection{The dual-invariant uniqueness theorem.}\label{secdiut}

\begin{prop}\label{giut}
Suppose that $(B,\delta)$ is a full coaction and $(\pi,\mu)$ is a covariant representation of $(B,\delta)$ in $M(C)$ such that $\ker\pi$ coincides with the kernel of the canonical map $j_B$ of $B$ into $M(B\times_\delta G)$. Suppose that there is an action $\beta:G\to \Aut C$ such that \begin{equation*}
\new{\overline{\beta_s}}(\pi(b))=\pi(b)\ \text{for $b\in B$}\ \text{ and }\ \new{\overline{\beta_s}}(\mu(f))=\mu(\rt_s(f))\ \text{for $f\in C_0(G)$.}
\end{equation*}
Then $\pi\times\mu:B\times_\delta G\to M(C)$ is injective.
\end{prop}

Proposition~\ref{giut} differs from Corollary~4.3 of \cite{rae} in two respects: it is about full coactions rather than reduced ones, and it is about representations with values in $C^*$-algebras rather than representations on Hilbert space. The second change is intended to make the result look more like the familiar gauge-invariant uniqueness theorems, and we begin with a similar rephrasing of the result in \cite{rae}.

\begin{lemma}\label{giutred}
Suppose that $(B,\epsilon)$ is a reduced coaction and $(\pi,\mu)$ is a covariant representation of $(B,\epsilon)$ in $M(C)$ such that $\pi$ is injective. Suppose that there is an action $\beta:G\to \Aut C$ such that
\begin{equation}\label{gaugeinvariance}
\new{\overline{\beta_s}}(\pi(b))=\pi(b)\ \text{for $b\in B$}\ \text{ and }\ \new{\overline{\beta_s}}(\mu(f))=\mu(\rt_s(f))\ \text{for $f\in C_0(G)$.}
\end{equation}
Then $\pi\times\mu$ is injective.
\end{lemma}

\begin{proof}
Choose a covariant representation $(\rho,U)$ of $(C,\beta)$ on a Hilbert space $\HH$ such that $\rho$ is faithful. (For example, $(\rho,U)$ could be the regular representation induced from a faithful representation of $C$.) Then the conditions \eqref{gaugeinvariance} imply that each $U_s$ commutes with each $\overline\rho\circ\pi(b)$ and that $(\overline\rho\circ\mu,U)$ is a covariant representation of the system $(C_0(G),\rt)$, or equivalently that $(\overline\rho\circ(\pi\times\mu),U)$ is a covariant representation of $(B\times_\epsilon G,\hat\epsilon)$. Since $\rho$ and $\pi$ are faithful, so is $\overline\rho\circ\overline{\pi\times\mu}\circ j_B=\overline\rho\circ\pi$. Now \cite[Corollary~4.3]{rae} implies that $(\overline\rho\circ\pi)\times(\overline\rho\circ\mu)$ is faithful, and since $(\overline\rho\circ\pi)\times(\overline\rho\circ\mu)=\overline\rho\circ(\pi\times\mu)$ by Lemma~\ref{barintform}, \new{it follows that $\pi\times\mu$ is faithful}.
\end{proof}

\begin{proof}[Proof of Proposition~\ref{giut}]
Since $\ker\pi=\ker j_B$, $\pi$ factors through an injective homomorphism $\pi^n:B^n:=B/\ker j_B\to M(C)$, and Proposition~\ref{propnormal}(b) implies that $(\pi^n,\mu)$ is a covariant representation of $(B^n,\delta^n)$.  Proposition~\ref{propred}(b) implies that $(\pi^n,\mu)$ is also a covariant representation of the reduction $(B^{nr},\delta^{nr})$. The automorphisms $\new{\overline{\beta_s}}$ leave everything in $\pi^n(B^n)=\pi(B)$ fixed, and we still have $\new{\overline{\beta_s}}(\mu(f))=\mu(\rt_s(f))$. Since $\delta^n$ is normal, $\delta^{nr}$ is a reduced coaction on the same algebra $B^n$.  Thus Lemma~\ref{giutred} implies that $\pi^n\times\mu$ is injective on $B^n\times_{\delta^{nr}}G$. Proposition~\ref{propred}(c)  implies that $(B^n\times_{\delta^{nr}}G,j_{B^n},j_G)$ is a crossed product for $(B^n,\delta^n)$, and Proposition~\ref{propnormal}(c) that $(B^n\times_{\delta^{nr}}G,j_{B^n}\circ q^n,j_G)$ is a crossed product for $(B,\delta)$. Since
\[
\pi^n\times\mu(j_{B^n}\circ q^n(b)j_G(f))=\pi^n(q^n(b))\mu(f)=\pi(b)\mu(f)=\pi\times \mu(j_B(b)\mu(f)),
\]
the representation of $B\times_\delta G$ corresponding to the representation $\pi^n\times\mu$ of $B^n\times_{\delta^{nr}}G$ is $\pi\times\mu$. Thus $\pi\times\mu$ is injective, as required.
\end{proof}

From Proposition~\ref{giut} we can deduce the following converse of Proposition~\ref{propnormal}(d).

\begin{cor}\label{convlem9}
Suppose that $(B,\delta)$ is a full coaction and $N$ is a closed normal subgroup of $G$ such that $(\delta|)^n=(\delta^n)|$. Then $j_B\times {j_G}|$ is injective.
\end{cor}

\begin{proof}Let $(j_B^{G/N},j_{G/N})$ and $(j_B^G,j_G)$ denote the universal covariant representations in $M(B\times_{\delta|}(G/N))$ and $M(B\times_\delta G)$.
Observe straightaway that since $\delta|^n$ is a coaction on $B/\ker j_B^{G/N}$ and $\delta^n|$ is a coaction on $B/\ker j_B^G$, we must have $\ker j_B^G=\ker j_B^{G/N}$.

Next, we apply Lemma~\ref{covrestricts} to $(j_B^G,j_G)$, obtaining a covariant representation $(j_B^G,j_G|)$ of $(B,G/N,\delta|)$ in $M(B\times_\delta G)$. The integrated form $j_B^G\times j_G|$ has range
\[
C:=\clsp\{j_B^G(b)j_G|(f):b\in B,\ f\in C_0(G/N)\}.
\]
The dual automorphisms $\hat\delta_s$ extend to automorphisms $\new{\overline{\hat\delta_s}}$ of $M(B\times_\delta G)$, and these automorphisms map $C$ into $C$. Since the action of $G$ by right translation on $C_0(G/N)$ is strongly continuous, so is $s\mapsto \new{\overline{\hat\delta_s}}:G\to \Aut C$. For $n\in N$, $\new{\overline{\hat\delta_n}}$ is the identity on $C$, and thus there is an action $\beta:G/N\to \Aut C$ such that $\beta_{sN}=\new{\overline{\hat\delta_s}|_C}$. This action satisfies $\new{\overline{\beta_{sN}}}(j_B^G(b))=j_B^G(b)$ and
\[
\new{\overline{\beta_{sN}}}(j_G|(f))=\new{\overline{\hat\delta_s}}(j_G|(f))=\new{\overline{\hat\delta_s}}(\new{\overline{j_G}}(f))=\new{\overline{j_G}}(\rt_s(f))=j_G|(\rt_{sN}(f)).
\]
Since $\ker j_B^G=\ker j_B^{G/N}$, applying Lemma~\ref{giutred} to $(j_B^G,j_G|)$ shows that $j_B^G\times j_G|$ is injective.
\end{proof}

\subsection{The different versions of Landstad's theorem}\label{secland}

The statement in part (a) of the following Proposition is Landstad's criterion for identifying reduced crossed products from \cite{L}, and (d) is the one used in \cite{kqcat}.

\begin{prop}\label{Landequiv}
Suppose that $B$ is a $C^*$-algebra and $G$ is a locally compact group. Then the following statements are equivalent:
\begin{itemize}
\item[(a)] there is a reduced coaction $\delta$ of $G$ on $B$ and a strictly continuous homomorphism $u:G\to UM(B)$ such that $\overline\delta (u_s)=u_s\otimes k_G^r(s)$ for $s\in G$;

\smallskip
\item[(b)] there is an object $(B,\delta,\phi)$ in the comma category $(C_r^*(G),\delta_G^r)\downarrow \cscoactndr(G)$;

\smallskip
\item[(c)] there is an object $(B,\epsilon,\psi)$ in $(C_r^*(G),\delta_G^n)\downarrow \cscoactndn(G)$;

\smallskip
\item[(d)] there is an object $(B,\epsilon,\theta)$ in $(C^*(G),\delta_G)\downarrow \cscoactndn(G)$.
\end{itemize}
\end{prop}

\begin{remark} Notice that in (d), the system $(C^*(G),\delta_G)$ is \emph{not} itself an object in the category $\cscoactndn(G)$ unless $G$ is amenable (because the normalisation is the dual coaction on the reduced $C^*$-algebra \cite[Proposition~A.61]{enchilada}). However, it is an object in the larger category $\cscoactnd(G)$ of full coactions, and this is enough for the comma category to make sense.
\end{remark}

\begin{proof}
For (b) $\Longrightarrow$ (a), let $(B,\delta,\new{\phi})$ be an object in $(C_r^*(G),\delta_G^r)\downarrow \cscoactndr(G)$.  Take $u:=\overline\phi\circ k_G^r$, which is strictly continuous because $\overline\phi$ is. Since $\phi$ is $\delta^r$\,--\,$\delta$ equivariant, that is, $\overline{\phi\otimes\id}\circ \delta_G^r=\overline\delta\circ\phi$, we have
\begin{align*}
\overline\delta(u_s)&=\overline\delta(\overline\phi(k_G^r(s))=\overline{\overline\delta\circ\phi}(k_G^r(s))=\overline{\overline{\phi\otimes\id}\circ\delta_G^r}(k_G^r(s))\\
&=\overline{\phi\otimes\id}\circ\overline{\delta_G^r}(k_G^r(s))=\overline{\phi\otimes\id}(k_G^r(s)\otimes k_G^r(s))\\
&=\overline\phi(k_G^r(s))\otimes\overline{\id}(k_G^r(s)) \quad \text{(using \eqref{barontensors}\,)}\\
&=u_s\otimes k_G^r(s).
\end{align*}

For (a) $\Longrightarrow$ (b),  suppose that we have $u$ as in (a). We have
\begin{equation*}\label{deltacircpi}
\overline{\overline\delta\circ\pi_u}(k_G(s))=\overline\delta\circ\new{\overline{\pi_u}}(k_G(s))=\overline\delta(u_s)=u_s\otimes k_G^r(s),
\end{equation*}
so $\overline\delta\circ\pi_u=\pi_{u\otimes k_G^r}$. We now choose a faithful nondegenerate representation $\tau$ of $B$ on a Hilbert space $H$. Since the integrated form $\pi_\lambda$ of the regular representation $\lambda:G\to U(L^2(G))$ factors through a faithful representation $\pi_\lambda^r$ of $C_r^*(G)$, we have a faithful nondegenerate representation $\tau\otimes \pi_\lambda^r$ of $B\otimes C_r^*(G)$ on $H\otimes L^2(G)$. Applications  of \eqref{barcomp} and \eqref{barontensors} show that the composition $\overline{\tau\otimes \pi_\lambda^r}\circ\pi_{u\otimes k_G^r}$ is the integrated form of $(\overline{\tau}\circ u)\otimes \lambda$. Since every representation of the form $V\otimes \lambda$ is unitarily equivalent to $1\otimes \lambda$, we have
\[
\ker(\pi_{u\otimes k_G^r})=\ker\big(\overline{\tau\otimes \pi_\lambda^r}\circ\pi_{u\otimes k_G^r}\big)=\ker\pi_{1\otimes\lambda}=\ker \pi_\lambda.
\]
Since $\delta$ and $\overline\delta$ are injective, we have
\[
\ker\pi_u=\ker(\overline\delta\circ\pi_u)= \ker(\pi_{u\otimes k_G^r})=\ker\pi_\lambda,
\]
and there is a homomorphism $\pi_u^r:C_r^*(G)\to M(B)$ such that $\pi_u=\pi_u^r\circ q^r$; it is nondegenerate because it has the same range as $\pi_u$. Now the calculation
\begin{align*}
\overline{\overline{\pi_u^r\otimes\id}\circ \delta_G^r}(k_G^r(s))
&=\overline{\pi_u^r\otimes\id}\circ \overline{\delta_G^r}(k_G^r(s))\quad \text{(using \eqref{barcompinC*}\,)}\\
&=\overline{\pi_u^r\otimes\id}(k_G^r(s)\otimes k_G^r(s))\\
&=\new{\overline{\pi_u^r}}(k_G^r(s))\otimes \new{\overline{\id}}(k_G^r(s))\quad\text{(using \eqref{barontensors}\,)}\\
&=u_s\otimes k_G^r(s)\\
&=\overline\delta(u_s)\\
&=\overline\delta\circ\new{\overline{\pi_u^r}}(k_G^r(s))\\
&=\overline{\overline\delta\circ\pi_u^r}(k_G^r(s))
\end{align*}
implies that $\overline{\pi_u^r\otimes\id}\circ \delta_G^r=\overline\delta\circ\pi_u^r$, which says that $(B,\delta,\pi_u^r)$ is an element of the comma category.

\smallskip

For (b) $\Longrightarrow$ (c),  suppose that $(B,\delta,\phi)$ belongs to the comma category in (b). Then Theorem~\ref{bashatbplate}(a) says that applying $\Red^{-1}$ to the reduced system $(B,\delta)$ and the morphism $\phi$ gives a normal system $(B,\quigg)$ and a morphism $\phi$ from $\Red^{-1}(C_r^*(G),\delta_G^r)$ to $(B,\quigg)$. Part (b) of Theorem~\ref{bashatbplate} implies that $\Red^{-1}(C_r^*(G),\delta_G^r)=(C_r^*(G),\delta_G^n)$, so $(B,\quigg,\phi)$ has the properties required of $(B,\epsilon,\psi)$ in (c).

To establish the implication (c) $\Longrightarrow$ (b), we apply $\Red$ to the system $(B,\epsilon,\psi)$ in (c), and use the properties of $\Red$ established in Theorem~\ref{bashatbplate}.

\smallskip

For (c) \new{$\Longleftrightarrow$} (d), note that
Theorem~\ref{bplatenor}(c) implies that the normalisation $(\delta_G)^n$ of the comultiplication on $C^*(G)$ is the full dual coaction $\pi_{k_G^r\otimes k_G}$ on the reduced group algebra $C_r^*(G)$ (which because of this is usually denoted $\delta_G^n$). In particular, we have $q^n=q^r$. If $(B,\epsilon,\theta)$ is as in (d), \new{then} $\theta^n$ is a morphism from $\Nor(C^*(G),\delta_G)=(C_r^*(G),\delta_G^n)$ to $\Nor(B,\epsilon)=(B,\epsilon)$, and $(B,\epsilon,\theta^n)$ is an element of the comma category in \new{(c)}. On the other hand, if $(B,\epsilon,\psi)$ is as in \new{(c)}, then $(B,\epsilon,\psi\circ q^r)=(B,\epsilon,\psi\circ q^n)$ belongs to the comma category in \new{(d)}.
\end{proof}

\subsection{Mansfield's imprimitivity theorem}
\label{subsecMansfield}

Mansfield's imprimitivity theorem \cite{man} is the analogue for crossed products by coactions of the imprimitivity theorem of Rieffel and Green for ordinary crossed products. As in \cite{hrman} and \cite{kqrproper}, we will approach Mansfield's theorem via Rieffel's general theory of proper actions \cite{proper, integrable}.

The version of Rieffel's theory used in \cite{hrman} and \cite{kqrproper} starts with a free and proper right action of $G$ on a locally compact space $T$, and the induced action $\rt$ of $G$ on $C_0(T)$. We then consider triples $(A,\alpha,\phi)$ in which $\alpha$ is an action of $G$ on a $C^*$-algebra $A$ and $\phi:C_0(T)\to M(A)$ is a homomorphism such that $\alpha_s\circ\phi=\phi\circ\rt_s$ (so that $(A,\alpha,\phi)$ is an object in the comma category $(C_0(T),\rt)\downarrow\csactnd(G)$). Rieffel proved in \cite[Theorem~5.7]{integrable} that $\alpha$ is proper  and saturated in the sense of \cite{proper} with respect to the subalgebra \[
A_0:=\phi(C_c(T))A\phi(C_c(T)):=\{\phi(f)a\phi(g):f,g\in C_c(G),\;a\in A\},
\]
which implies that $A_0$ completes to give a Morita equivalence $Z(A,\alpha,\phi)$ between $A\times_{\alpha,r}G$ and a generalised fixed-point algebra $\Fix(A,\alpha,\phi)$ sitting in $M(A)$.

To get Mansfield's theorem for a reduced coaction $(B,\delta)$ and a closed subgroup $H$ of $G$, we apply Rieffel's results with $A=B\times_\delta G$, $\alpha=\hat\delta|_H$, $(T,G)=(G,H)$ and $\phi=j_G:C_0(G)\to M(B\times_\delta G)$. The crucial point is that Rieffel's fixed-point algebra turns out to be the crossed product by the homogeneous space $G/H$, which is the $C^*$-subalgebra
\[
B\times_{\delta,r}(G/H):=\clsp\{j_B(b)j_G|(f):b\in B, f\in C_0(G/H)\}
\]
of $M(B\times_\delta G)$. (The equality of $\Fix(B\times_\delta G,\hat\delta|,j_G)$ and $B\times_{\delta,r}(G/H)$ follows from Theorem~3.1 of \cite{hrman} by an argument given at the beginning of \cite[\S6]{kqrproper}. We stress that this equality is highly nontrivial, and its proof relies heavily on intricate calculations in Mansfield's original paper \cite{man}. Indeed, even the assertion that $B\times_{\delta,r}(G/H)$ is a $C^*$-subalgebra appears to be deep.)

We want to apply our boilerplate to  Mansfield's imprimitivity theorem, and we need to define crossed products by homogeneous spaces for full coactions. For a normal coaction $(C,\epsilon)$, there is no problem: with our conventions, $(C\times_{\epsilon} G,j_C,j_B)$ is exactly the same as $(C\times_{\epsilon^r} G,j_C,j_B)$, so we can define  $C\times_{\epsilon,r}(G/H)$ to be $C\times_{\epsilon^r,r}(G/H)$. Note that $C\times_{\epsilon,r}(G/H)$ is a $C^*$-subalgebra of $M(C\times_\epsilon G)$.

\begin{lemma}\label{lemG/H}
Let $H$ be a closed subgroup of a locally compact group $G$, and suppose that $(D,\eta)$ is a full coaction of $G$. Then the extension  to the multiplier algebra of the isomorphism $q^n\times G:D\times_{\eta}G\to D^n\times_{\eta^n}G$ of Corollary~\ref{natisonor} maps $\clsp\{j_D(d)j_G|(f):b\in B, f\in C_0(G/H)\}$ onto $D^n\times_{\eta^n,r}(G/H)$.
\end{lemma}

\begin{proof} Recall that $q^n\times G$ is by definition $(j_{D^n}\circ q^n)\times j_G$, so we have
\[
\overline{q^n\times G}(j_D(d)j_G|(f))=j_{D^n}(q^n(d))j_G^{D^n}|(f)
\]
and the result follows because $\overline{q^n\times G}$ is an isomorphism.
\end{proof}

It follows from Lemma~\ref{lemG/H} that $\clsp\{j_D(d)j_G|(f):b\in B, f\in C_0(G/H)\}$ is a $C^*$-subalgebra of $M(D\times_\eta G)$, and it makes sense to define the crossed product by
\[
D\times_{\eta,r}(G/H):=\clsp\{j_D(d)j_G|(f):b\in B, f\in C_0(G/H)\}.
\]
Then Lemma~\ref{lemG/H} says that $\overline{q^n\times G}$ restricts to an isomorphism $q^n\times_r(G/H)$ of $D\times_{\eta,r}(G/H)$ onto $D^n\times_{\eta^n,r}(G/H)$.

\begin{thm}\label{mansfield1}
Let $H$ be a closed subgroup of a locally compact group $G$.
\smallskip

\textnormal{(a)} For every reduced coaction $(B,\delta)$, $Z(B\times_\delta G,\hat\delta|,j_G)$ is a $(B\times_\delta G)\times_{\hat\delta|,r}H$\,-\,$B\times_{\delta,r}(G/H)$ imprimitivity bimodule.

\smallskip
\textnormal{(b)}  For every normal coaction $(C,\epsilon)$, $Z(C\times_\epsilon G,\hat\epsilon|,j_G)$ is a $(C\times_\epsilon G)\times_{\hat\delta|,r}H$\,-\,$C\times_{\delta,r}(G/H)$ imprimitivity bimodule \new{that} coincides with $Z(C\times_{\epsilon^r} G,\hat{\epsilon^r}|,j_G)$.
\smallskip

\textnormal{(c)}  Suppose that $(D,\eta)$ is a full coaction. Then the isomorphism $q^n\times G:D\times_{\eta}G\to D^n\times_{\eta^n}G$ of Corollary~\ref{natisonor} maps
\[
A_0:=j_G(C_c(G))(D\times_{\eta} G)j_G(C_c(G))\ \text{ onto }\ B_0:=j_G(C_c(G))(D^n\times_{\eta^n} G)j_G(C_c(G)),
\]
and extends to an isometric map $\psi$ of $Z(D\times_{\eta} G,\hat\eta, j_G)$ onto $Z(D^n\times_{\eta^n} G,\hat{\eta^n}, j_G)$ such that $((q^n\times G)\times_rH,\psi,q^n\times_r(G/H))$ is an isomorphism of imprimitivity bimodules.
\end{thm}

\begin{remark}
When the subgroup $H$ is normal, a coaction $(B,\delta)$ of $G$ restricts to a coaction $\delta|$ of the quotient $G/H$, and one would prefer a version of Mansfield's theorem \new{that} uses the crossed product $B\times_{\delta|} (G/H)$ by this coaction (as opposed to the crossed product $B\times_{\delta,r} (G/H)$ by the homogeneous space appearing in Theorem~\ref{mansfield1}).

For a full coaction, $\delta|:=\overline{{\id}\otimes \pi_q}\circ\delta$, and the crossed products $B\times_{\delta|} (G/H)$ and $B\times_{\delta,r} (G/H)$ may differ. If the map $j_B\times j_G|:B\times_{\delta|}(G/H)\to M(B\times_\delta G)$ is injective (which is automatic if $\delta$ is normal \cite[Lemma~3.2]{KQimprimitivity}), then it is an isomorphism of $B\times_{\delta|} (G/H)$ onto $B\times_{\delta,r} (G/H)$, and we obtain the Morita equivalence between $(B\times_\delta G)\times_{\hat\delta}H$ and $B\times_{\delta|}(G/H)$ established in \cite[Corollary~3.4]{KQimprimitivity}. For a reduced coaction $(B,\delta)$, the restriction $\delta|$ is by definition the restriction $\delta^Q|$ of the Quiggification; since $\delta^Q$ is normal, we again get an isomorphism of $B\times_{\delta|}(G/H)$ onto $B\times_{\delta,r} (G/H)$. The resulting Morita equivalence was first obtained in \cite[Proposition~4.2]{hrman} and is a direct generalisation of Mansfield's original theorem \cite{man}. See \cite[\S4]{hrman} for further details.

However, we believe that Theorem~\ref{mansfield1}(c) itself is new.
\end{remark}

\begin{proof}[Proof of Theorem~\ref{mansfield1}]
As we observed above, part (a) was deduced in \cite{kqrproper} from \cite[Theorem~3.1]{hrman}. Part (b) is now easy: with the definition of $C\times_{\epsilon,r}(G/H)$ we have given and the conventions for crossed products we are using, the coefficient algebras are exactly the same, and hence so is the bimodule.  Part (c) follows from the Lemma~\ref{lem-general} below by taking $A=D\times_\eta G$, $\alpha=\hat\eta|$, $\phi_A=j_G^D$, $B=D^n\times_{\eta^n} G$, $\beta=\new{\widehat{\eta^n}}|$, $\phi_B=j_G^{D^n}$, $(T,G)=(G,H)$ and $\rho=q^n\times G$.
\end{proof}

\begin{lemma}\label{lem-general}
Let $G$ be a locally compact group  and let $\rho:(A,\alpha,\phi_A)\to (B,\beta, \phi_B)$ be an isomorphism in the comma category $(C_0(T),\rt)\downarrow\csactnd(G)$. Then $\rho$ maps $A_0:=\phi_A(C_c(T))A\phi_A(C_c(T))$ onto $B_0:=\phi_B(C_c(T))B\phi_B(C_c(T))$ and extends to an isometric map $\psi$ of $Z(A,\alpha,\phi_A)$ onto $Z(B,\beta,\phi_B)$ such that $(\rho\times_r G,\psi,\rho|)$ is an isomorphism of imprimitivity bimodules.
\end{lemma}

\begin{proof}
That $\rho$ is an  isomorphism in the comma category means that it is an $\alpha$\,--\,$\beta$ equivariant isomorphism of $A$ onto $B$ satisfying $\overline\rho\circ\phi_A=\phi_B$. Thus $\rho(A_0)=B_0$. The right inner product $\langle a,b\rangle_R$ of $a,b\in A_0$ is the multiplier of $A$ such that
\[
c\cdot\langle a,b\rangle_R=\int_G c\alpha_s(a^*b)\,ds\ \text{ for every $c\in A_0$},
\]
and since $\rho$ is $\alpha$\,--\,$\beta$ equivariant, the multiplier $\overline{\rho}(\langle a,b\rangle_R)$ of $B$ has the property \new{that} characterises $\langle \rho(a),\rho(b)\rangle_R$. Thus $\rho$ is inner-product preserving, and extends to an isometry $\psi$ of the completion $Z(A,\alpha,\phi_A)$ of $A_0$ onto $Z(B,\beta,\phi_B)=\overline{B_0}$.

For $a,b,c\in A_0$ we have
\begin{align*}
\psi(a)\cdot\overline\rho(\langle b\,,\, c\rangle_R)&=\rho(a)\overline\rho(\langle b\,,\, c\rangle_R)=\rho(a\langle b\,,\, c\rangle_R)
=\psi(a\langle b\,,\, c\rangle_R)=\psi(a\cdot \langle b\,,\, c\rangle_R),
\end{align*}
so $\psi$ is a right-module homomorphism. Since $\rho\times_r G$ acts pointwise, for $f\in C_c(G,A_0)$ and $a\in A_0$ we have
\begin{align*}
\psi(f\cdot a)&=\psi\Big( \int f(s)\alpha_s(a)\Delta(s)^{1/2}\, ds \Big)=\rho\Big( \int f(s)\alpha_s(a)\Delta(s)^{1/2}\, ds \Big)\\
&=\int\rho(f(s))\beta_s(\rho(a))\Delta(s)^{1/2}\, ds =\int\rho\times_rG(f)(s)\beta_s(\rho(a))\Delta(s)^{1/2}\, ds \\
&=\rho\times_r H(f)\cdot\psi(a),
\end{align*}
so $\psi$ is a left-module homomorphism. A standard calculation now shows that $\psi$ preserves the left inner product, and hence is an isomorphism of imprimitivity bimodules.
\end{proof}

\subsection{Naturality of Mansfield imprimitivity}\label{secmannat} There is a second category $\cscat$ of $C^*$-algebras in which the imprimitivity bimodules are the isomorphisms, and if we view the various crossed product functors as taking values in $\cscat$, then it makes sense to ask whether the imprimitivity bimodules in Mansfield's theorem give a natural isomorphism. This was proved  for reduced coactions in Theorem~6.2 of \cite{kqrproper}. 

Just to be precise: the objects in $\cscat$ are $C^*$-algebras, and the morphisms from one $C^*$-algebra $A$ to another $B$ are isomorphism classes $[{}_AX_B]$ of full right-Hilbert $A$--$B$ bimodules. It was shown in \cite[\S2]{taco} that $\cscat$ is a category with composition defined by $[{}_AX_B][{}_BY_C]=[{}_A(X\otimes_B Y)_C]$, and that $[{}_AX_B]$ is an isomorphism if and only if $X$ is an imprimitivity bimodule.  Every nondegenerate homomorphism $\sigma:A\to M(B)$ gives a morphism $[\sigma]$ with underlying right Hilbert module $B_B$. 

To see what the naturality result in \cite{kqrproper} says, consider a morphism $\sigma:(B,\delta)\to (C,\epsilon)$ in $\cscoactndr(G)$. This induces nondegenerate morphisms  $\sigma\times G:=(\new{\overline{j_C}}\circ\sigma)\times j_G^C$ from $B\times_\delta G$ to $M(C\times_\epsilon G)$ and $(\sigma\times G)\times_r H=(i_{C\times G}\circ(\sigma\times G))\times i_G^{C\times G}$. Then Theorem~6.2 of \cite{kqrproper} says that the diagram
\begin{equation}
  \xymatrix@C=8pc@R=3pc{(B\times_\delta G)\times_{\hat\delta|,r}H\ar[r]^-{Z(B\times_\delta G,\hat\delta|,j^B_G)}
    \ar[d]_{(\sigma\times G)\times_r H} & B\times_{\delta,r}(G/H)\ar[d]^{(\sigma\times G)|} \\
    (C\times_\epsilon G)\times_{\hat\epsilon|,r}H\ar[r]_-{Z(C\times_\epsilon G,\hat\epsilon|,j^C_G)}&C\times_{\epsilon,r}(G/H)}
\end{equation}
commutes in $\cscat$, or in other words that the right-Hilbert bimodule
\[
Z(B\times_\delta G,\hat\delta|,j^B_G)\otimes_{ B\times_{\delta,r}(G/H)}\big(C\times_{\epsilon,r}(G/H)\big)\]
is isomorphic to
 \[\big((C\times_\epsilon G)\times_{\hat\epsilon|,r}H\big)\otimes_{(C\times_\epsilon G)\times_{\hat\epsilon|,r}H}Z(C\times_\epsilon G,\hat\epsilon|,j^C_G).
\]

We can now formulate our naturality result.

\begin{thm}\label{thm-naturality}
Let $H$ be a closed subgroup of a locally compact group $G$. Then Rieffel's bimodules $Z(B\times_\delta G, \hat\delta|, j_G)$ implement a natural isomorphism  between the functors $(B,\delta)\mapsto B\times_{\delta,r}(G/H)$ and $(B,\delta)\mapsto (B\times_\delta G)\times_{\hat\delta|,r}H$ from  \textnormal{(a)} $\cscoactndr(G)$ to $\cscat$, \textnormal{(b)} $\cscoactndn(G)$ to $\cscat$ and \textnormal{(c)} $\cscoactnd(G)$ to $\cscat$.
\end{thm}

\begin{proof}[Proof of Theorem~\ref{thm-naturality}]
As mentioned above, (a) is Theorem~6.2 of \cite{kqrproper}. Part (b) follows immediately from (a) because, with our conventions, all the algebras, homomorphisms and bimodules for $(B,\delta)$ are exactly the same as they are for $(B,\delta^r)$.

For (c) we fix a morphism $\sigma:(B,\delta)\to (C,\epsilon)$, and let $\sigma^n$ be the unique morphism from $(B^n,\delta^n)$ to $(C^n,\epsilon^n)$ such that $\new{\overline{q^n}}\circ \sigma=\sigma^n\circ q^n$ (which we know exists by Lemma~\ref{uniquepropnormal}). Now we consider the following diagram.
\begin{equation*}
  \xymatrix{(B^{n}\times_{\delta^{n}}G)\times_{\new{\widehat{\delta^{n}}},r}H
\ar[rrr]^-{Z(B^{n}\times_{\delta^{n}}G)}\ar[ddd]_{(\sigma^{n}\times
  G)\times_r H}
&&& B^{n}\times_{\delta^{n},r}(G/H)\ar[ddd]^{(\sigma^{n}\times G)|}\\
&(B\times_{\delta}G)\times_{\hat\delta,r}H\ar[r]^-{\strut Z(B\times_\delta G)}
\ar[ul]_{(q^{n}\times G)\times_r H}\ar[d]_{(\sigma\times G)\times_r H}&
B\times_{\delta,r}(G/H)\ar[ru]^{q^{n}\times_r (G/H)}
\ar[d]^{(\sigma\times
G)|}\\
&(C\times_{\epsilon}G)\times_{\hat\epsilon,r}H\ar[r]^-{\strut Z(C\times_\epsilon
  G)}\ar[dl]_{(q^{n}\times G)\times_r H} &C\times_{\epsilon,r}
(G/H)\ar[dr]^{q^{n}\times_r(G/H)}\\
(C^{n}\times_{\epsilon^{n}}G)\times_{\new{\widehat{\epsilon^{n}}},r}H
\ar[rrr]^{Z(C^{n}\times_{\epsilon^{n}}G)}&&& C^{n}\times_{\epsilon^{n},r}(G/H).}
\end{equation*}
In this diagram, we know from (b)  that the outside square commutes, and we want to prove that the inside square commutes. If $(\phi,\theta,\psi):{}_AX_B\to {}_CY_D$ is an isomorphism of imprimitivity bimodules, then $x\otimes d\mapsto \theta(x)\cdot d$ is an isomorphism of ${}_A(X\otimes_BD)_D$ onto ${}_A(C\otimes_CY)_D={}_AY_D$, and hence Theorem~\ref{mansfield1}(c) implies that the top and bottom rectangles commute. The compositions in the right-hand quadrilateral are the restrictions of $\overline{\sigma^n\times G}\circ \overline{q^n\times G}$ and $\overline{q^n\times G}\circ \overline{\sigma\times G}$, so to see that that quadrilateral commutes it suffices for us to see that $(\sigma^n\times G)\circ(q^n\times G)=\overline{q^n\times G}\circ(\sigma\times G)$, which follows immediately from the functoriality of the crossed-product construction. A similar argument shows that the left-hand quadrilateral commutes.
\new{Since the arrows connecting the inside and outside squares are isomorphisms, it follows that the inside square commutes.}
\end{proof}


\appendix

\section{Barring and tensor products}\label{secbarring}

For every
nondegenerate homomorphism $\phi:A\to M(B)$, there is a unital homomorphism
$\overline \phi:M(A)\to M(B)$ such that $\overline\phi|_A=\phi$ (see \cite[Corollary~2.51]{tfb}, for
example). The extension $\overline\phi$ has to satisfy
$\overline\phi(m)(\phi(a)b)=\phi(ma)b$ for $m\in M(A)$, $a\in A$ and $b\in B$; this equation implies  that there is exactly one such extension, and that $\overline\phi$ is strictly continuous. The uniqueness implies the identity
\begin{equation}\label{barcompinC*}
\overline{\overline\phi\circ\psi}=\overline\phi\circ\overline\psi,
\end{equation}
which is often used but seldom mentioned.

Suppose that $\phi:A\to M(B)$ and $\psi:C\to M(D)$ are nondegenerate homomorphisms. Then there is a unique homomorphism $\phi\otimes\psi$  from the spatial tensor product $A\otimes C$ to the spatial tensor product $M(B)\otimes M(D)$ such that $\phi\otimes \psi(a\otimes c)=\phi(a)\otimes \psi(c)$ (see \cite[Proposition~B.13]{tfb}, for example). Although this homomorphism $\phi\otimes \psi$ takes values in a unital algebra, and hence in a multiplier algebra, it need not be nondegenerate (think of $C_0(X)\otimes C_0(X)$ going into $C_b(X)\otimes C_b(X)$), and hence cannot obviously be extended to $M(A\otimes C)$. To get round this, we compose $\phi\otimes\psi$ with the inclusion $\iota:M(B)\otimes M(D)\to M(B\otimes D)$ characterised by $\iota(m\otimes n)(b\otimes d)=(mb)\otimes (nd)$. (If $\pi:B\to B(\HH)$ and $\rho:D\to B(\K)$ are faithful, then the representation $\overline\pi\otimes\overline\rho$ is a faithful representation of $M(B)\otimes M(D)$ with range contained in $\overline{\pi\otimes\rho}(M(B\otimes D)$ (using \cite[B.11 and 2.53]{tfb}), so the homomorphism $\iota$ is well-defined and injective.)  Now one can easily check (see the next proposition) that $\iota\circ(\phi\otimes\psi):A\otimes C\to M(B\otimes D)$ is nondegenerate, and therefore has a unique extension to $M(A\otimes C)$. The extension $\overline{\phi\otimes\psi}$ \new{that} appears in the coaction literature, either explicitly or implicitly, is really $\overline{\iota\circ(\phi\otimes\psi)}$; the $\iota$ itself never appears. (We use the same symbol $\iota$ for any pair of $C^*$-algebras because we are just trying to emphasize that these maps are there.)

The properties of these extensions established in the next proposition are used without comment in the literature.

\begin{prop}\label{apponbarring} \textnormal{(a)} Suppose that $\phi:A\to M(B)$ and $\psi:C\to M(D)$ are nondegenerate homomorphisms. Then the composition
\[
\iota\circ(\phi\otimes\psi):A\otimes C\to M(B)\otimes M(D)\to M(B\otimes D)
\]
is a nondegenerate homomorphism, and its extension  to $M(A\otimes C)$ satisfies
\begin{equation}\label{barontensorsiota}
\overline{\iota\circ(\phi\otimes\psi)}(\iota(m\otimes n))=\iota\big(\overline\phi(m)\otimes\overline\psi(n)\big)\ \text{ for $m\in M(A)$, $n\in M(C)$.}
\end{equation}

\smallskip
\textnormal{(b)} Suppose that $\phi_i:A_i\to M(B_i)$ and $\psi_i:B_i\to M(C_i)$ are nondegenerate homomorphisms for $i=1$ and $i=2$. Then
\begin{equation}\label{barcompiota}
\iota\circ\big((\new{\overline{\psi_1}}\circ\phi_1)\otimes(\new{\overline{\psi_2}}\circ\phi_2)\big)=\overline{\iota\circ(\psi_1\otimes\psi_2)}\circ(\iota\circ(\phi_1\otimes \phi_2)).
\end{equation}

\smallskip
\textnormal{(c)} Suppose that $\phi:A\to M(B)$ is a nondegenerate homomorphism, $C$ is another $C^*$-algebra and define $(\phi\otimes 1)(a):=\phi(a)\otimes 1$. Then $\iota\circ(\phi\otimes 1)$ is nondegenerate and
\begin{equation*}
\overline{\iota\circ(\phi\otimes 1)}=\iota\circ(\overline\phi\otimes 1).
\end{equation*}
\end{prop}

\begin{proof}
The nondegeneracy of $\phi$ and $\psi$ and the Cohen factorisation theorem imply that every elementary tensor in $B\otimes D$ has the form
\[
\phi(a)b\otimes \psi(c)d=\iota(\phi(a)\otimes\psi(c))(b\otimes d)=\big(\iota\circ(\phi\otimes \psi)(a\otimes c)\big)(b\otimes d),
\]
so $\iota\circ(\phi\otimes\psi)$ is nondegenerate. To establish \new{\eqref{barontensorsiota}}, we consider the effect of the multipliers on an elementary tensor $\phi(a)b\otimes \psi(c)d$. We have:
\begin{align*}
\overline{\iota\circ(\phi\otimes\psi)}(\iota(m\otimes n))&\big(\phi(a)b\otimes \psi(c)d\big)\\
&=\overline{\iota\circ(\phi\otimes\psi)}(\iota(m\otimes n))\big(\iota\circ(\phi\otimes \psi)(a\otimes c)\big)(b\otimes d)\\
&=\iota\circ(\phi\otimes\psi)\big(\iota(m\otimes n)(a\otimes c)\big)(b\otimes d)\\
&=\iota\circ(\phi\otimes\psi)(ma\otimes nc)(b\otimes d)\\
&=\iota(\phi(ma)\otimes \psi(nc))(b\otimes d)\\
&=\phi(ma)b\otimes \psi(nc)d\\
&=\overline\phi(m)\phi(a)b\otimes \overline\psi(n)\psi(c)d\\
&=\iota(\overline\phi(m)\otimes \overline\psi(n))(\phi(a)b\otimes\psi(c)d),
\end{align*}
and this gives \new{\eqref{barontensorsiota}}. For \new{\eqref{barcompiota}}, we plug an elementary tensor $a_1\otimes a_2$ into the left-hand side and apply it to an elementary tensor $\psi_1(b_1)c_1\otimes\psi_2(b_2)c_2$:
\begin{align*}
\iota\circ\big((&\new{\overline{\psi_1}}\circ\phi_1)\otimes(\new{\overline{\psi_2}}\circ\phi_2)\big)(a_1\otimes a_2)\big(\psi_1(b_1)c_1\otimes\psi_2(b_2)c_2\big)\\
&=\iota\big(\new{\overline{\psi_1}}\circ\phi_1(a_1)\otimes\new{\overline{\psi_2}}\circ\phi_2(a_2)\big)\big(\psi_1(b_1)c_1\otimes\psi_2(b_2)c_2\big)\\
&=\big(\new{\overline{\psi_1}}\circ\phi_1(a_1)\psi_1(b_1)c_1\big)\otimes\big(\new{\overline{\psi_2}}\circ\phi_2(a_2)\psi_2(b_2)c_2\big)\\
&=\psi_1(\phi_1(a_1)b_1)c_1\otimes\psi_2(\phi_2(a_2)b_2)c_2\\
&=\iota\circ(\psi_1\otimes\psi_2)\big(\phi_1(a_1)b_1\otimes\phi_2(a_2)b_2)\big)(c_1\otimes c_2)\\
&=\iota\circ(\psi_1\otimes\psi_2)\big(\iota\circ(\phi_1\otimes\phi_2)(a_1\otimes a_2)(b_1\otimes b_2)\big)(c_1\otimes c_2)\\
&=\overline{\iota\circ(\psi_1\otimes\psi_2)}\big(\iota\circ(\phi_1\otimes\phi_2)(a_1\otimes a_2)\big)\big(\iota\circ(\psi_1\otimes\psi_2)(b_1\otimes b_2)(c_1\otimes c_2)\big)\\
&=\overline{\iota\circ(\psi_1\otimes\psi_2)}\big(\iota\circ(\phi_1\otimes\phi_2)(a_1\otimes a_2)\big)\big(\psi_1(b_1)c_1\otimes \psi_2(b_2)c_2\big),
\end{align*}
which does what we want. For (c), we note that nondegeneracy is easy, let $m\in M(A)$, $a\in A$, and compute:
\begin{align*}
\overline{\iota\circ(\phi\otimes 1)}(m)\big(\iota\circ(\phi\otimes 1)(a)\big)&=\iota\circ(\phi\otimes 1)(ma)=\iota(\phi(ma)\otimes 1)\\
&=\iota(\overline\phi(m)\phi(a)\otimes 1)=\iota(\overline\phi(m)\otimes 1)\iota(\phi(a)\otimes 1)\\
&=\iota\circ(\overline\phi\otimes 1)(m)\iota\circ(\phi\otimes 1)(a).\qedhere
\end{align*}
\end{proof}

\begin{sugcon*}
Using $\overline{\phi\otimes\psi}$ to denote $\overline{\iota\circ(\phi\otimes\psi)}$ is not too dangerous, since there is nothing else that $\overline{\phi\otimes\psi}$ could mean. When we drop the bar altogether, though, we introduce ambiguities:  $\phi\otimes\psi$ could mean $\overline\phi\otimes\overline\psi$ or $\phi\otimes\overline\psi$ or $\overline{\phi\otimes\psi}$ or many other things. Worse, there are calculations in the literature  where we need to switch meanings. Since ambiguity is the enemy, and since deciding exactly how to make our  calculations precise seems to be a rather slippery matter, we think it  best to explicitly use the bar and apply Proposition~\ref{apponbarring} as we go.

So below and in the main sections of this paper, the inclusions $\iota$ of $M(C)\otimes M(D)$ in $M(C\otimes D)$ are silent, but we try to make the extensions to multiplier algebras explicit. 
\end{sugcon*}

\begin{cor} Suppose that $\phi_i:A_i\to M(B_i)$ and $\psi_i:B_i\to M(C_i)$ are nondegenerate homomorphisms for $i=1$ and $i=2$.  Then
\begin{align}\label{barontensors}
\overline{\phi_1\otimes\phi_2}(m\otimes n)&=\new{\overline{\phi_1}}(m)\otimes\new{\overline{\phi_2}}(n)\ \text{ for $m\in M(A_1)$, $n\in M(A_2)$,}\\
\label{barcomp}
(\new{\overline{\psi_1}}\circ\phi_1)\otimes(\new{\overline{\psi_2}}\circ\phi_2)\big)&=\overline{\psi_1\otimes\psi_2}\circ(\phi_1\otimes \phi_2),\\
\label{phiotimes1}
\overline{\phi_1\otimes 1}&=\new{\overline{\phi_1}}\otimes 1,\\
\label{N1}\overline{(\new{\overline{\psi_1}}\circ\phi_1)\otimes(\new{\overline{\psi_2}}\circ\phi_2)}
&=\overline{\psi_1\otimes\psi_2}\circ\overline{\phi_1\otimes \phi_2},\quad\text{and}\\
\label{N2}\overline{\psi_1\otimes\id}\circ\overline{\id\otimes\phi_2}&=\overline{\id\otimes\phi_2}\circ\overline{\psi_1\otimes\id}.
\end{align}
\end{cor}

\begin{proof}
The first three equations are just the results in Proposition~\ref{apponbarring} with our new convention. To establish \eqref{N1}, we use  \eqref{barcompinC*} and then \eqref{barcomp}:
\begin{align*}
\overline{\psi_1\otimes\psi_2}\circ\overline{\phi_1\otimes \phi_2}
&=\overline{
\overline{\psi_1\otimes \psi_2}\circ (\phi_1\otimes\phi_2)}=\overline{(\overline{\psi_1}\circ\phi_1)\otimes(\overline{\psi_2}\circ\phi_2)}.
\end{align*}
Now taking $\psi_2=\id$ and $\phi_1=\id$ in \eqref{N1} gives
\[
\overline{\psi_1\otimes\id}\circ\overline{\id\otimes \phi_2}=\overline{(\new{\overline{\psi_1}}\circ\phi_1)\otimes(\new{\overline{\id}}\circ\phi_2)}=
\overline{\psi_1\otimes\phi_2}.
\]
Similarly $\overline{\id\otimes \phi_2}\circ\overline{\psi_1\otimes\id}=\overline{\psi_1\otimes\phi_2}$, and we have proved \eqref{N2}.
\end{proof}

We now illustrate how these formulas are used by proving three standard results which are needed in the main text.

\begin{lemma}\label{covrestricts}
Suppose that $(B,G,\delta)$ is a full coaction and $N$ is a closed normal subgroup of $G$. If $(\pi,\mu)$ is a covariant representation of $(B,G,\delta)$ in $M(C)$, then $(\pi,\mu|)$ is a covariant representation of $(B,G/N,\delta|)$.
\end{lemma}

\begin{proof}
We write $\iota_N$ for the nondegenerate embedding of $C_0(G/N)$ in $M(C_0(G))$, so that $\mu|:=\overline\mu\circ\iota_N$. For $f\in C_0(G,C^*(G/N))=C_0(G)\otimes C^*(G/N)$,
\[
\overline{{\id}\otimes \pi_q}(w_G)({\id}\otimes \pi_q)(f)=({\id}\otimes \pi_q)(w_Gf)
\]
is the function $s\mapsto \pi_q(w_G(s)f(s))=w_{G/N}(sN)\pi_q(f(s))$. Thus, viewed as a multiplier of $C_0(G, C^*(G/N))$, $\overline{{\id}\otimes \pi_q}(w_G)$ is multiplication by the function $s\mapsto w_{G/N}(sN)$. This function is the image under $\overline{\iota_N\otimes\id}$ of the function $w_{G/N}\in M(C_0(G/N,C^*(G/N)))$, and thus we have
\begin{align}
\overline{\mu\otimes\id}\circ\overline{{\id}\otimes \pi_q}(w_G)
&=\overline{\mu\otimes\id}\circ\overline{\iota_N\otimes\id}(w_{G/N})\notag\\
&=\overline{(\overline\mu\circ\iota_N)\otimes\id}(w_{G/N})\notag\quad\text{(using \eqref{N1}\,)}\\
&=\overline{\mu|\otimes\id}(w_{G/N}).\label{relwGwGN}
\end{align}
Now we take $b\in B$ and compute, obtaining
\begin{align*}
\overline{\pi\otimes\id}(\delta|(b))
&=\overline{\pi\otimes\id}\circ\overline{{\id}\otimes {\pi_q}}(\delta(b))\\
&=\overline{{\id}\otimes\pi_q}\circ\overline{\pi\otimes\id}(\delta(b))\quad\text{(using \eqref{N2}\,)}\\
&=\overline{{\id}\otimes \pi_q}\circ \Ad\overline{\mu\otimes\id}(w_G)(\pi(b)\otimes 1_{C^*(G)})\\
&= \Ad\big(\overline{{\id}\otimes \pi_q}\circ\overline{\mu\otimes\id}(w_G)\big)(\pi(b)\otimes 1_{C^*(G/N)})\\
&= \Ad\big(\overline{\mu\otimes\id}\circ\overline{{\id}\otimes \pi_q}(w_G)\big)(\pi(b)\otimes 1_{C^*(G/N)})\quad\text{(using \eqref{N2}\,)}\\
&= \Ad\big(\overline{\mu|\otimes\id}(w_{G/N})\big)(\pi(b)\otimes 1_{C^*(G/N)})\quad\text{(using \eqref{relwGwGN}\,)},
\end{align*}
which is the required covariance.
\end{proof}

\begin{lemma}\label{barintform}
Suppose that $\alpha:G\to \Aut A$ is an action, that $(\pi,u)$ is a covariant representation of $(A,\alpha)$ in $M(B)$ such that $\pi\times u$ factors through a representation $\pi\times_r u$ of $A\times_{\alpha,r} G$, and that $\theta:B\to M(C)$ is a nondegenerate homomorphism. Then $\overline\theta\circ(\pi\times u)$ also factors through a representation $(\overline\theta\circ\pi)\times_r(\overline\theta\circ u)$ of $A\times_{\alpha,r} G$, and we have $\overline\theta\circ(\pi\times_r u)=(\overline\theta\circ\pi)\times_r(\overline\theta\circ u)$.
\end{lemma}

\begin{proof}
Using \eqref{barcompinC*}, we have
\[
\overline{\overline\theta\circ(\pi\times u)}\circ i_A=\overline\theta\circ\overline{\pi\times u}\circ i_A=\overline\theta\circ\pi
\]
and $\overline{\overline\theta\circ(\pi\times u)}\circ i_G=\overline\theta\circ u$. Hence
$\overline\theta\circ(\pi\times u)=(\overline\theta\circ\pi)\times(\overline\theta\circ u)$. Now we  compute:
\[
\overline\theta\circ(\pi\times_r u)\circ q^r=\overline\theta\circ(\pi\times u)=(\overline\theta\circ\pi)\times(\overline\theta\circ u).
\]
Since the left-hand side factors through $q^r$, so does the right-hand side, and then by definition of $\times_r$ the right-hand side is $((\overline\theta\circ\pi)\times_r(\overline\theta\circ u))\circ q^r$.
\end{proof}

The next proposition was implicitly asserted on page 420 of \cite{kqcat}.

\begin{prop}\label{ex1} 
Let $G$ be a locally compact group. There is a category $\cscoactnd(G)$ in which the objects are full coactions of $G$ on $C^*$-algebras, the morphisms from $(A,\delta)$ to $(B,\epsilon)$ are the nondegenerate homomorphisms $\phi:A\to M(B)$ such that $\overline{\phi\otimes\id}\circ\delta=\overline\epsilon\circ\phi$, and $\psi\circ \phi$ is by definition $\overline\psi\circ\phi$.
\end{prop}

\begin{proof}
We know from Proposition~1 of \cite{aHRWsurvey} that there is a category $\cscatnd$ of $C^*$-algebras and nondegenerate homomorphisms in which composition is defined by $\psi\circ \phi:=\overline\psi\circ\phi$. So we need to prove that if $\phi:A\to M(B)$  is $\delta$\,--\,$\epsilon$ equivariant  $\psi:B\to M(C)$ is $\delta$\,--\,$\eta$ equivariant, then
\begin{equation}\label{compequiv}
\overline{(\overline\psi\circ\phi)\otimes\id}\circ\delta=\overline\eta\circ({\overline\psi\circ\phi}).
\end{equation}

Applying \eqref{barcomp} with $\phi_1=\phi_2$ the identity on $C^*(G)$, we find that
\[
(\overline\psi\circ\phi)\otimes\id=\overline{\psi\otimes\id}\circ(\phi\otimes\id).
\]
We now plug this into the left-hand side of \eqref{compequiv}, and use \eqref{N1} and \eqref{barcompinC*}:
\begin{align*}
\overline{(\overline\psi\circ\phi)\otimes\id}\circ\delta
&=\big(\overline{\psi\otimes\id}\circ\overline{\phi\otimes\id}\big)\circ\delta=\overline{\psi\otimes\id}\circ\big(\overline{\phi\otimes\id}\circ\delta\big)\\
&=\overline{\psi\otimes\id}\circ\overline\epsilon\circ{\phi}=\overline{\overline{\psi\otimes\id}\circ\epsilon}\circ{\phi}\\
&=\overline{\overline\eta\circ\psi}\circ{\phi}=\overline\eta\circ\overline{\psi}\circ{\phi}.\qedhere
\end{align*}
\end{proof}

\end{document}